\documentclass[12pt]{amsart}

\usepackage[utf8]{inputenc}
\usepackage[english]{babel}

\usepackage{csquotes}
\usepackage{enumitem}
\usepackage{amsmath, amsthm, amsfonts, amssymb}
\usepackage{mathtools}
\usepackage[colorlinks=true, linkcolor=red, citecolor=green, backref=page]{hyperref}
\usepackage{color}
\usepackage{graphicx}
\usepackage[capitalise, noabbrev]{cleveref}
\usepackage{geometry}
\usepackage{marginnote}
\usepackage{tikz}
\usepackage{tikz-cd}
\usetikzlibrary{decorations.markings}
\usepackage{centernot}
\usepackage{chngcntr}
\usepackage{mathrsfs}
\usepackage{placeins}
\usepackage{stmaryrd}

\usepackage{algorithmic}

\newcommand{\dd}{\mathrm{d}}

\newcommand{\xrightarrowdbl}[2][]{\xrightarrow[#1]{#2}\mathrel{\mkern-14mu}\rightarrow}

\newcommand\restr[2]{{
  \left.\kern-\nulldelimiterspace 
  #1 
  \vphantom{\big|} 
  \right|_{#2} 
  }}

\DeclareMathOperator{\im}{Im}

\DeclareMathOperator{\DR}{DR}
\DeclareMathOperator{\Gr}{Gr}
\DeclareMathOperator{\Der}{Der}
\DeclareMathOperator{\codim}{codim}

\newcommand{\derR}{\mathbf{R}}
\newcommand{\derL}{\mathbf{L}}

\newcommand{\Ltensor}{\overset{\derL}{\otimes}}
\newcommand{\decal}[1]{\lbrack #1 \rbrack}

\newcommand{\Omod}{\mathscr{O}}
\newcommand{\Dmod}{\mathscr{D}}
\newcommand{\Mmod}{\mathcal{M}}

\newcommand{\shH}{\mathscr{H}}
\newcommand{\shB}{\mathscr{B}}
\newcommand{\shI}{\mathscr{I}}
\newcommand{\shK}{\mathscr{K}}
\newcommand{\shG}{\mathscr{G}}
\newcommand{\shC}{\mathscr{C}}
\newcommand{\shS}{\mathcal{S}}
\newcommand{\shF}{\mathscr{F}}
\newcommand{\shM}{\mathscr{M}}
\newcommand{\shN}{\mathscr{N}}
\newcommand{\shCC}{\underline{\mathbb{C}}}

\newcommand{\shExt}{\mathscr{E}\hspace{-1.5pt}\mathit{xt}}

\newcommand{\shHom}{\mathscr{H}\hspace{-2.5pt}\mathit{om}}

\newcommand{\ZZ}{\mathbb{Z}}
\newcommand{\QQ}{\mathbb{Q}}
\newcommand{\RR}{\mathbb{R}}
\newcommand{\CC}{\mathbb{C}}

\newcommand{\jump}{\vskip 2mm}

\numberwithin{equation}{section}

\theoremstyle{plain}
\newtheorem{theorem}{Theorem}
\newtheorem{proposition}{Proposition}
\newtheorem{lemma}{Lemma}
\newtheorem{corollary}{Corollary}

\theoremstyle{definition}

\newtheorem{remark}{Remark}
\newtheorem{example}{Example}
\AtBeginEnvironment{example}{\pushQED{\qed}}
\AtEndEnvironment{example}{\popQED\endexample}

\newcommand{\GM}{\shG^\bullet_f}

\numberwithin{theorem}{subsection}
\numberwithin{proposition}{subsection}
\numberwithin{lemma}{subsection}
\numberwithin{corollary}{subsection}
\numberwithin{remark}{subsection}
\numberwithin{equation}{subsection}
\numberwithin{example}{subsection}

\newcounter{intro}

\newtheorem{intro-proposition}[intro]{Proposition}
\newtheorem{intro-corollary}[intro]{Corollary}
\newtheorem{intro-theorem}[intro]{Theorem}
\newtheorem{intro-question}[intro]{Question}

\title[The Gauss-Manin system of an ICIS]{The Gauss-Manin system of an ICIS}

\author[G. Blanco]{Guillem Blanco}

\thanks{The author is supported by a Postdoctoral Fellowship of the Research Foundation -- Flanders}

\address{Department of Mathematics\\ KU Leuven,
Celestijnenlaan 200B, 3001 Leuven, Belgium.}
\email{guillem.blanco@kuleuven.be}

\begin{document}

\begin{abstract}
For an isolated complete intersection singularity (ICIS), we define and study its Gauss-Manin system and its associated Hodge filtration. We show the relation between the Hodge filtration and a generalized Brieskorn lattice and study conditions for the existence of a microlocal structure. Using these ideas, we relate the residue on the saturated Brieskorn lattice with a \(b\)-function introduced by Torrelli.
\end{abstract}

\maketitle

\renewcommand\thesubsection{\arabic{subsection}}

\subsection{Introduction}

Let \(f : (\CC^{n+1}, x) \longrightarrow (\CC, 0)\) be a germ of a holomorphic function defining an isolated hypersurface singularity. The Gauss-Manin system is the complex of \(\Dmod\)-modules defined by the \(\Dmod\)-module theoretic direct image by \(f\) of the structure sheaf. The cohomology sheaves of this complex are regular holonomic \(\Dmod\)-modules that carry a natural good filtration, the Hodge filtration. The Gauss-Manin system is directly connected with the Picard-Lefschetz monodromy, the classical theory of the Gauss-Manin connection, the Brieskorn lattice, and the spectrum of the singularity. It has been extensively studied in \cite{Pham79, Sai82, SS85, Sai89}.

\jump

In this work, we define and study the Gauss-Manin system \(\shG^\bullet_f\) of an isolated complete intersection singularity (ICIS) \((Z, x)\) by means of 1-parametric smoothings \(f : (X, x) \longrightarrow (\CC, 0)\) and the local cohomology module associated with \(X\), see \cref{sec:GM} for the precise definition and basic properties. In contrast to the hypersurface case, the natural good filtration in the ICIS case comes from the order-Ext filtration in local cohomology, and the Hodge filtration is harder to describe, see \cref{sec:filtrations}. Our first main result is the following.

\begin{theorem} \label{thm:main}
Let \((Z, x)\) be an ICIS of dimension \(n > 0\), then
\[
    \shH^0 \shG_f^\bullet \simeq \shH^n(f_* \omega_f^\bullet) \otimes_{\CC} \CC[\partial_t],
\]
and the order-Ext filtration \(F_k\) coincides with the filtration by the order of \(\partial_t\).
\end{theorem}

Here \(\omega_f^\bullet\) denotes the complex of relative regular differential forms; see \cref{sec:relative-forms}. The absolute version of this complex \(\omega_{X}^\bullet\) was introduced by Barlet \cite{Bar78} and Kersken \cite{Ker84}, independently. The stalk at \(x \in X\) of \(\shH^n(f_* \omega_f^\bullet)\) coincides with the generalized Brieskorn lattice \(H''_{f,x}\) introduced by Greuel \cite{Gre75}, see \cref{prop:cohomology-relative}. As a corollary of \cref{thm:main} and the results in \textit{loc.~cit.}, one can deduce that \(\shH^0 \shG_f^\bullet\) is a regular holonomic \(\Dmod\)-module, see \cref{sec:coh-reg}.

\jump

In general, Hodge filtration \(F^H_k\) in \(\shH^0 \shG_f^\bullet\) is only contained in the order-Ext filtration \(F_k\). It is a recent result due to Musta\c{t}\v{a} and Popa \cite{MP22}, that the coincidence at level zero of both filtrations in local cohomology is equivalent to the singularities being Du Bois. If we call a smoothing Du Bois when \((X, x)\) has Du Bois singularities, we have

\begin{proposition} \label{prop:hodge}
   If \(f : (X, x) \longrightarrow (\CC, 0)\) is a Du Bois smoothing of a positive-dimensional ICIS \((Z, x)\), then 
   \[
    F^H_0 \shH^0 \shG_f^\bullet \simeq \shH^n(f_* \omega_f^\bullet).
   \]
\end{proposition}

Since Du Bois singularities deform, if \((Z, x)\) has Du Bois singularities, any smoothing is Du Bois. However, having a Du Bois smoothing is more general than the exceptional fiber being Du Bois; for instance, any hypersurface singularity has a Du Bois smoothing. 

\jump

Another important property of the Brieskorn lattice in the hypersurface case is that it has a microlocal structure, that is, \(\shH^n(f_* \omega^\bullet_f)\) is a module over the ring \(\CC\{\{\partial_t^{-1}\}\}\) of microlocal differential operators with constant coefficients. For an ICIS, we can show that this property is related to smoothings having rational singularities.

\begin{proposition} \label{prop:microlocal}
    If \(f : (X, x) \longrightarrow (\CC, 0)\) is a rational smoothing of an ICIS \((Z, x)\), then the Brieskorn lattice \(\shH^n(f_* \omega_f^\bullet)\) has structure of \(\CC\{\{\partial_t^{-1}\}\}\)-module.
\end{proposition}

Rational singularities are Du Bois, so the same relation between the singularities of \((Z, x)\) and \((X, x)\) holds. However, it follows from an inversion of adjunction result due to Schwede \cite{Sch07} that, in this setting, if the exceptional fiber has Du Bois singularities, then any smoothing is rational. Therefore, it is enough to assume that \((Z, x)\) is Du Bois for both \cref{prop:hodge,prop:microlocal} to hold.

\jump

Finally, we show the following relation between the Gauss-Manin system and the theory of \(b\)-functions. Assume \(X\) is embedded in a complex manifold \(U\), and let \(F \in \Omod_{U,x}\) such that \(\restr{F}{X} = f\). Consider the (local) \(b\)-function \(b_{\delta, f}(s)\) defined by the functional equation
\[
    P(s) \delta F^{s+1} = b_{\delta, f}(s) \delta F^s,
\]
where \(\delta\) is the local cohomology class of degree zero defined by a set of generators of the analytic space \(X\). Denote by \(\tilde{b}_{\delta, f}(s) = b_{\delta, f}(s) / (s + 1)\) the reduced \(b\)-function. These \(b\)-functions were studied by Torrelli in \cite{Tor05}. However, we will make use of a third polynomial \(\widehat{b}_{\delta, f}(s)\) sitting between these two that was introduced in \cite{Tor09}, see \cref{sec:bfunc} for the precise definition.

\begin{theorem} \label{thm:malgrange}
    Let \(f : (X, x) \longrightarrow (\CC, 0)\) be a rational smoothing of an ICIS \((Z, x)\). Then the minimal polynomial of \(-\partial_t t\) in \(\widetilde{H}_{f,x}'' / t \widetilde{H}_{f,x}''\) is equal to \(\widehat{b}_{\delta, f}(s)\).
\end{theorem}

Here \(\widetilde{H}_{f,x}''\) denotes the saturation of \(H''_{f,x}\) with respect to \(t\partial_t\), see \eqref{eq:saturation}. In particular, the degree of \(b_{\delta, f}(s)\) is bounded by the Milnor number \(\mu(Z, x)\) + 1. \cref{thm:malgrange} generalizes a classical result due to Malgrange \cite{Mal75} in the hypersurface case. 

\jump

The relation of the results presented in this work with the spectrum and the minimal exponent of an ICIS \((Z, x)\) will be treated in a future work.

\jump

This work is organized as follows. In \cref{sec:icis} we review basic results and properties of ICIS. The Gauss-Manin system is defined in \cref{sec:GM}. The relation with the generic monodromy of an ICIS is discussed in \cref{sec:generic-monodromy}. Good filtrations on local cohomology are treated on \cref{sec:graph,sec:filtrations}. The definition and properties of absolute and relative differentials forms are the content of \cref{sec:regular-forms,sec:relative-forms,sec:normal,sec:regular-icis}. \cref{thm:main} and its consequences are proven in \cref{sec:cohomology,sec:coh-reg}. The role of rational and Du Bois singularities is explored in \cref{sec:vfilt,sec:rational}. Finally, for relation with the theory of \(b\)-functions see \cref{sec:bfunc}.

\subsection*{Acknowledgements} We would like to thank J. \`Alvarez-Montaner, D. Bath, and M. Torrecillas for useful discussion during the development of this work. We are indebted to M. Musta\c{t}\v{a} for pointing out some mistakes in an early version of this work and the inversion of adjunction result from \cite{Sch07}.

\subsection{Isolated complete intersection singularities} \label{sec:icis}

Let \((Z, x)\) be an isolated complete intersection singularity of dimension \(n\), or ICIS for short. Let \(f : (X, x) \longrightarrow (\CC^k, 0)\) be a flat morphism defining a deformation of \((Z, x)\), where \((X, x)\) is a germ of an analytic space. The flatness of \(f\) and the fact that the base of the deformation is smooth imply that \((X, x)\) is also a complete intersection of dimension \(n + k\), see \cite[6.10]{Loo84}. 

\jump

Let \(f : X \longrightarrow S\) be a morphism of analytic spaces representing the above morphism of germs. Denote by \(X_t = f^{-1}(t) \), \(t \in S\), the fibers of \(f\). A deformation \(f : (X, x) \longrightarrow (\CC^k, 0)\) of an ICIS is called a \(k\)-parameter smoothing if there exists a representative such that the generic fibers \(X_t, t \in S \setminus \{0\}\), are smooth. Notice that, by Sard's theorem, this is equivalent to \(X\) having only an isolated singularity at \(x\). A 1-parametric smoothing will simply be called a smoothing of \((Z, x)\).

\jump

We can always embed \(X\) in a germ of a complex manifold \(U\). Denote by \(F : U \longrightarrow S\) any holomorphic function such that \(F|_X = f\). We will always denote \(r = \codim_U X\) so that the dimension \(d\) of \(U\) equals \(n + r + k\). 

\jump

Assume that for each point of \(X_0 \setminus \{x\}\), \(X\) is not singular and \(F|_X\) is a submersion. Choose \(r : X \longrightarrow \RR_{\geq 0}\) a nonnegative real analytic function such that \(r^{-1}(0) \cap X_0 = \{x\}\). Such a function \(r\) is said to define the point \(x \in X\). For any \(\epsilon \in \RR_{\geq 0}\), let
\[
B_{\epsilon} = \{ y \in X \ |\ r(y) < \epsilon \}, \qquad S_{\epsilon} = \{ y \in X\ |\ r(y) = \epsilon \}.
\]
Choosing \(\epsilon > 0 \) small enough and shrinking \(S\) in such a way that it is contractible, the morphism
\[
f : X \cap B_{\epsilon} \cap f^{-1}(S) \longrightarrow S
\]
is called a good representative of the germ if \(S_\epsilon\) intersects \(f^{-1}(t)\) transversally, for all \(s \in S\), and \(X_0\) intersects \(S_\eta\) transversally, for all \(0 < \eta \leq \epsilon\). 

\jump

If \(\shC_f\) denotes the set of critical points of any good representative, and \(D_f = f(\shC_f) \subseteq S\) the discriminant of \(f\), the restriction of any good representative on the complement of the critical locus defines a smooth fiber bundle on the complement of the discriminant, see \cite[2.8]{Loo84}. This fiber bundle is usually known as the Milnor fibration, and the generic fiber \(X_t, t \in S \setminus D_f,\) as the Milnor fiber.

\jump

The existence of versal deformations implies that Milnor fibers from different smoothings are diffeomorphic, and the Milnor fiber is always homotopy equivalent to a bouquet of spheres. The number \(\mu(Z, x) = \dim_{\CC} H_{n}(X_t, \CC)\) is called the Milnor number of \((Z, x)\). It can be computed algebraically by means of the L\^e-Greuel formula \cite[5.3]{Gre75}, \cite[3.6.4]{Tra74}: for any smoothing \(f : (X, x) \longrightarrow (\CC, 0)\) of an ICIS \((Z, x)\), one has
\begin{equation} \label{eq:le-greuel}
\mu(X, x) + \mu(Z, x) = \nu(D_f, 0)
\end{equation}
where \(\nu(D_f, 0)\) is the multiplicity of the discriminant of \(D_f\) at the origin.

\jump

The restricted morphism \(f|_{\shC_f} : (\shC_f, x) \longrightarrow (\CC, 0)\) is a finite morphism. As a consequence, the discriminant \(D_f\) is a hypersurface in \(S\), see \cite[4.8]{Loo84}, and it can be endowed with a non-reduced structure in such a way that it is compatible with base change, see \cite[\S 4.E]{Loo84}.

\subsection{The Gauss-Manin system} \label{sec:GM}

When \(f : U \longrightarrow S\), with \(U\) smooth, is a good representative of an isolated hypersurface singularity, the Gauss-Manin system is defined as the \(\Dmod\)-module theoretic direct image of the structure sheaf \(\Omod_U\) along \(f\), see \cite[\S 15]{Pham79}. Similarly, in the case where \(f : X \longrightarrow S\) is a good representative of a smoothing of an ICIS, we define the Gauss-Manin system as the complex of \(\Dmod_S\)-modules,
\begin{equation}\label{eq:GM}
    \GM = F_+ \shH^r_{X}(\Omod_U),
\end{equation}
where \(F : U \longrightarrow S\) is a holomorphic function with \(U\) smooth such that \(F|_X = f\). Denoting by \(i : X \xhookrightarrow{\quad} U\) the closed embedding, by \cite[1.1]{Meb76}, the Gauss-Manin system satisfies
\[
    \DR_S\GM \simeq \derR F_* \DR_U \shH^r_{X} (\Omod_U) \simeq \derR F_* i_* \shCC_{X}[r] \simeq \derR f_* \shCC_{X}[r],
\]
in the derived category \(D^b(\underline{\CC}_S)\). We define the \(k\)-th Gauss-Manin system \(\shH^k_f\) as the \(k\)-th cohomology sheaf \(\shH^k \GM\) of \(\GM\). We will mostly care about \(\shH^0_f\). The intrinsic nature of \(\shG_f^\bullet\) will become clear after \cref{prop:cohomology-derham}.

\begin{lemma}
Assume that \(F : U \longrightarrow S\) is a locally trivial fibration, \(X \subseteq U\) is a complex manifold of codimension \(r\) and denote \(f : X \longrightarrow S\) the restriction of \(F\) to \(X\). Then 
\[
\shG_f^{\bullet} \simeq (\derR f_* \shCC_X \otimes_{\shCC} \Omod_S)[n].
\]
\end{lemma}
\begin{proof}
Since \(F\) is a submersion, 
\[
    F_+ \shH^r_{X}(\Omod_U) = \derR F_*\big(\Dmod_{U \leftarrow S} \Ltensor_{\Dmod_U} \shH^r_{X}(\Omod_U)\big) \simeq \derR F_* \big(\DR_{U/S} \shH_{X}^r(\Omod_U)\big)
\]
by \cite[14.4.1]{Pham79}, where \(\DR_{U/S}\) denotes the relative de Rham complex. Denote \(\iota : X \xhookrightarrow{\quad} U\) the closed embedding of complex manifolds. The fact that \(\shH^r_{X}(\Omod_U) \simeq \iota_+ \iota^* \Omod_U \), see \cite[5.1]{MaTo04}, gives
\[
\DR_{U/S}\shH^r_{X}(\Omod_U) \simeq \DR_{U/S}(\iota_+ \iota^* \Omod_U) \simeq \iota_* \DR_{X/S} \Omod_X,
\]
by a relative version of \cite[4.2.5]{HTT08}. If the relative codimension of \(f\) is \(n\), by definition \(\DR_{X/S}\Omod_X = \Omega^{n + \bullet}_{X/S} \), so
\[
    \GM \simeq \derR F_* \iota_* \DR_{X/S} \Omod_X \simeq \derR f_* \Omega^{\bullet}_{X/S}\decal{n}.
\]
A relative version of the Poincar\'e Lemma implies that \(\Omega^\bullet_{X/S}\) is a resolution of \(f^{-1} \Omod_S\), see \cite[8.4]{Loo84}, hence
\[
    \derR f_* \Omega^\bullet_{X/S} \simeq \derR f_* f^{-1} \Omod_S \simeq \derR f_* \shCC_X \otimes_{\shCC} \Omod_S,
\]
where the last isomorphism follows from \cite[8.1]{Loo84}.
\end{proof}

\subsection{Generic monodromy} \label{sec:generic-monodromy}

Given a smoothing \(f : (X, x) \longrightarrow (\CC, 0)\) of an ICIS, we have a well-defined action of the fundamental group of \(\CC^*\) on the Milnor fibration. This gives a monodromy endomorphism on the Milnor fiber and its cohomology groups. However, in contrast with the hypersurface case, the Milnor fibrations defined by two smoothings are, in general, not isomorphic. In particular, the corresponding monodromy endomorphisms can have different eigenvalues.

\jump

Let \((Z, x)\) be an ICIS of dimension \(n\). For all \(i \geq 0\) define \(\mu_i(Z, x)\) to be the minimal Milnor number of an ICIS \((X_i, x)\) such that there exists a deformation \(f_i : (X_i, x) \longrightarrow (\CC^i, 0)\).  Then, we say that a \(k\)-parameter deformation \(f_k : (X_k, x) \longrightarrow (\CC^k, 0)\) of the ICIS \((Z, x)\) is called \(\mu_k\)-minimal if \(\mu(X_k, x) = \mu_k(Z, x)\). A \(\mu_1\)-minimal deformation of \((Z, x)\) will be called a \(\mu\)-minimal smoothing. Notice that \((X_1, x)\) is smooth, i.e. \(\mu_1(Z, x) = 0\), if and only if \((Z, x)\) is a hypersurface singularity.

\jump

It follows from \cite[1]{Para91} and the L\^e-Greuel formula \eqref{eq:le-greuel} that the Milnor fibrations of any two \(\mu\)-minimal smoothings are equivalent. Hence, the monodromy of a \(\mu\)-minimal smoothing only depends on \((Z, x)\), and we will call this the \emph{generic monodromy} of the ICIS. The generic monodromy of an ICIS is related to Verdier monodromy, see \cite[1.3.\textit{i}]{DMS11}. However, in general, their eigenvalues are not equal, see \cref{ex:2}.

\jump

A \(\mu\)-minimal smoothing can be constructed by restricting the semiuniversal deformation \(\Phi : (\mathscr{X}, x) \longrightarrow (S, 0)\) of \((Z, x)\) to a general line in the base \(S\) that intersects the discriminant \(D_\Phi\) transversally at the origin.

\begin{example}
Let \((Z, 0) \subseteq (\CC^3, 0)\) be the ICIS defined by the ideal \(I = (F_1, F_2)\), with 
\[F_1 = z^2 - xy, \qquad F_2 = x^2 + y^2 + z^2.\]
It has Milnor number \(\mu(Z, 0) \) and Tjurina number \(\tau(Z, 0) \) both equal to 5. Consider \(\Phi : (\mathscr{X}, 0) \longrightarrow (\CC^5, 0) \) the semiuniversal deformation of \((Z, 0)\), where \((\mathscr{X}, 0) \subseteq (\CC^3 \times \CC^5, 0) \) is defined by
\[ G_1 = z^2 -xy + t_0 y + t_1, \qquad G_2 = x^2 + y^2 + z^2 + t_2 z + t_3 y + t_4, \]
and \(\Phi\) is the restriction to \((\mathscr{X}, 0)\) of the projection to the second component of \(\CC^3 \times \CC^5\), having coordinates \(t_0, \dots, t_4\). The discriminant \((D_\Phi, 0) \subseteq (\CC^5, 0)\) of the semiuniversal deformation has multiplicity \(6\), and its tangent cone equals
\[
t_1^6 - 2t_1^5t_4 + \frac{1}{2}t_1^4t_4^2 + t_1^3t_4^3 - \frac{7}{16}t_1^2t_4^4 - \frac{1}{8}t_1t_4^5 + \frac{1}{16}t_4^6.
\]
If we consider now a general line in \((\CC^5, 0)\), it can be parameterized as
\[
\varphi : (\CC^1, 0) \longrightarrow (\CC^5, 0), \qquad t \mapsto (\alpha_0 t, \alpha_1 t, \alpha_2 t, \alpha_3 t, \alpha_4 t),
\]
for generic \(\alpha_0, \dots, \alpha_4 \in \CC\). The induced deformation \(\varphi^* \Phi : \varphi^*(\mathscr{X}, 0) \longrightarrow (\CC, 0)\) by the base change \(\varphi\) is a smoothing of \((Z, 0)\). By the invariance of the discriminant under base change, the tangent cone of its discriminant \((D, 0) \subseteq (\CC, 0)\) is \(P(\alpha_0, \dots, \alpha_4) t^6\) where,
\[ P(\alpha_0, \dots, \alpha_4) = \alpha_1^6 - 2\alpha_1^5\alpha_4 + \frac{1}{2}\alpha_1^4\alpha_4^2 + \alpha_1^3\alpha_4^3 - \frac{7}{16}\alpha_1^2\alpha_4^4 - \frac{1}{8}\alpha_1\alpha_4^5 + \frac{1}{16}\alpha_4^6. \]
The formula of L\^e-Greuel \eqref{eq:le-greuel} gives,
\[
    \mu(\varphi^*(\mathscr{X}, 0)) + \mu(Z, 0) = \nu(D, 0),
\]
so the Milnor number of \(\varphi^*(\mathscr{X}, 0)\) is 1 and \(\varphi^*\Phi\) is \(\mu_1\)-minimal smoothing of \((Z, 0)\) if and only if \(P(a_0, \dots, a_4) \neq 0\). 
\end{example}

In practice, to construct a \(\mu\)-minimal smoothing, it is enough to restrict any holomorphic function
\begin{equation} \label{eq:holomorphic-function}
f : (\CC^{n+r}, x) \longrightarrow (\CC^r, 0)
\end{equation}
defining \((Z, x)\) to a line in \(\CC^r\) intersecting the discriminant \(D_f\) of \(f\) transversally at the origin. This follows from the next lemma, which is probably already known, but we could not find a reference for it in the literature.

\begin{lemma}
Let \(\nu(Z, x)\) be the multiplicity of the discriminant of the semiuniversal deformation of \((Z, x)\). Then the discriminant of any holomorphic function \eqref{eq:holomorphic-function} defining \((Z, x)\) has also multiplicity \(\nu(Z, x)\).
\end{lemma}

\begin{proof}
Let \(f_1, \dots, f_r\) be the components of the morphism in \eqref{eq:holomorphic-function}. The total space of a semiuniversal deformation of \((Z, x)\) is given by the variety defined by
\[
    F_i = f_i + \sum_{j = 0}^{\tau} m_{i, j} t_j, \quad \textnormal{for} \quad i = 1, \dots, r,
\]
in \(\CC^{n+r} \times \CC^\tau\), where \(m_{i, j}\) form a monomial basis of the vector space \(T^1_{(Z, x)}\) of fist order deformations of \((Z, x)\). The projection from the second component defines the semiuniversal deformation of \((Z, x)\). The total space of the restriction to a general line in \(\CC^\tau\) is given by
\[
    G_i = f_i + u_i(x) t, \quad \textnormal{for} \quad i = 1, \dots, r,
\]
in \(\CC^{n+r} \times \CC\), where \(u_i = u_i(x), i = 1, \dots, r\), are units. However, the equations \(f_i + u_i t \) and \(u_i^{-1} f_i + t\) define the same germ at \(x\), so the multiplicity of the discriminant of the projection to the coordinate \(t\) is the same, and it equals \(\nu(Z, x)\) by construction.

\jump

By versality and invariance of the discriminant under base change, the multiplicity of the discriminant of the projection to \(t_1, \dots, t_r\) from \(u_i^{-1} f_i + t_i = 0, i = 1, \dots, r,\) must coincide with \(\nu(Z, x)\).  Since \(f_i\) and \(u_i^{-1} f_i\) define the same function germ, the same is true for the discriminant corresponding to \(f_i + t_i = 0, i = 1, \dots, r\); in other words, the discriminant of \eqref{eq:holomorphic-function} has multiplicity \(\nu(Z, x)\).
\end{proof}

\jump

Through the rest of this work we can always assume that the smoothings considered are \(\mu\)-minimal, and they are obtained from \eqref{eq:holomorphic-function} by considering the projection to \(t\) from the variety defined by \(f_i - \alpha_i t = 0, i = 1, \dots, r\), with \(\alpha_i \in \CC\) generic.

\subsection{The graph embedding} \label{sec:graph}
To compute a direct image, one typically factors the morphism \(F : U \longrightarrow S\) into a closed embedding and a projection, namely
\[
    F : U \xrightarrow{\ \ \iota\ \ } U \times S \xrightarrow{\ \ p\ \ } S,
\]
where \(\iota : U \longrightarrow U \times S, x \mapsto (x, F(x))\) is the graph embedding, and \( p : U \times S \longrightarrow S\) is the projection to the second component.

\begin{lemma} \label{lemma:graph-embedding}
With the notations above,
\begin{equation} \label{eq:graph-embedding}
    \iota_{+} \shH^r_{X}(\Omod_U) \simeq \shH^{r + 1}_{\Gamma_f}(\Omod_{U \times S}),
\end{equation}
where \(\Gamma_f\) denotes the graph of \(f : X \longrightarrow S \) in \(U \times S\).
\end{lemma}
\begin{proof}
Working on the derived category \(D^b(\Dmod_U)\) of bounded complexes of left \(\Dmod_U\)-modules, we identify \(\shH^r_{X}(\Omod_U)\) with \(\derR \Gamma_{[X]} \Omod_U \decal{r}\). Writing \(\Omod_{U} = \iota^* \Omod_{U \times S}\) and using the fact that the direct image and the local cohomology functors commute, see \cite[4.1]{MaTo04},
\[
\iota_{+} \derR \Gamma_{[X]} \Omod_U \decal{r} = \iota_{+} \derR \Gamma_{[X]} \derL \iota^* \Omod_{U \times S} \decal{r} \simeq \iota_{+} (\derL \iota^* \derR \Gamma_{\Gamma_f} \Omod_{U \times S} \decal{r}),
\]
since \(\iota(X) = \Gamma_f\). For any object \(\Mmod \) in \(D^b(\Dmod_{U \times S})\) and any submanifold \(i : Y \hookrightarrow U \times S\) of codimension \(q\), one has \(i_{+} (\derL i^* \Mmod) \simeq \derR \Gamma_{[Y]} \Mmod \decal{q}\), see \cite[5.1]{MaTo04}. Hence,
\[
\iota_{+} \derR \Gamma_{[X]} \Omod_U \decal{r} \simeq \derR \Gamma_{\Gamma_F} (\derR \Gamma_{\Gamma_f} \Omod_{U \times S} \decal{r})\decal{r + 1} \simeq \derR \Gamma_{\Gamma_F\, \cap\, \Gamma_f} \Omod_{U \times S} \decal{r + 1}.
\]
The result follows by taking cohomology at degree zero and noticing that \(\Gamma_f \subset \Gamma_F\).
\end{proof}

Alternatively, we can identify the direct image by the graph embedding of \(\shH^{r}_{X}(\Omod_U)\) with the \(\Dmod_{U \times S}\)-module 
\begin{equation} \label{eq:graph-embedding2}
    \shB_f = \shH^r_{X} (\Omod_U) \otimes_{\CC} \CC[\partial_{t}],
\end{equation}
with the \(\Dmod_{U \times S}\)-action given as follows. Let  \(m \in \shH^r_{X}(\Omod_U), g \in \Omod_{U}\), and \(\xi \in \Der(\Omod_{U})\), then
\begin{equation} \label{eq:graph-action}
\begin{split}
    \partial_{t}  (m \otimes \partial_{t}^\nu) = m \otimes \partial_t^{\nu+1}, &\qquad \xi (m \otimes \partial_{t}^\nu) = (\xi m) \otimes \partial_t^\nu - (\xi F) m \otimes \partial_t^{\nu+1}, \\
    g (m \otimes \partial_t^\nu) = g m \otimes \partial_t^\nu, &\qquad t (m \otimes \partial_t^\nu) = F m \otimes \partial_t^\nu - \nu m \otimes \partial_t^{\nu-1}.
\end{split}
\end{equation}
One then has \(\GM \simeq \derR p_*\DR_{U \times S/S}\shB_f\), where \(\DR_{U \times S/S}\shB_f = \Omega^{d + \bullet}_{U \times S/S} \otimes_{\Omod_{U\times S}} \shB_f\) is the relative de Rham complex for the projection \(p\).

\jump

This description of the Gauss-Manin system can be further simplified, see \cite[\S 3]{SS85}. Define the complex of \(F^{-1}\Dmod_S\)-modules
\begin{equation} \label{eq:complex-K}
    \shK^\bullet = \shH^r_{X}(\Omega_U^{d+\bullet}) \otimes_{\CC} \CC[\partial_t]
\end{equation}
placed in degrees \(-d, \dots, 0\), with differential
\begin{equation} \label{eq:diff-complex-K}
    \underline{\dd} (\omega \otimes \partial_t^k)  = \dd\omega \otimes \partial_t^k - (\dd F \wedge \omega) \otimes \partial_t^{k+1},
\end{equation}
and \(F^{-1}\Dmod_S\)-action as in \eqref{eq:graph-action}. The complex \((\shH^r_{X}(\Omega_U^{d+\bullet}), \dd) \) can be identified with the de Rham complex of the \(\Dmod_U\)-module \(\shH^r_{X}(\Omod_U)\). Then
\[
    \iota^{-1} \DR_{U \times S/S}\shB_f \simeq \shH^r_{X}(\Omega^{d + \bullet}_U) \otimes \CC[\partial_t],
\]
and since \(\DR_{U\times S/S}\shB_f\) is supported on the graph of \(f\),
\begin{equation} \label{eq:gm-derf}
    \GM \simeq \derR F_* (\iota^{-1} \DR_{U \times S/S}\shB_f) \simeq \derR f_*\shK^\bullet.
\end{equation}

\subsection{Filtrations on local cohomology} \label{sec:filtrations}

The only non-vanishing local cohomology sheaf \(\shH^r_{X}(\Omod_U)\) of a local complete intersection \(X\) comes equipped with two natural filtrations. The order filtration on \(\shH^r_{X}(\Omod_U)\) is the increasing filtration given by
\[
    O_k \shH^r_{X} (\Omod_U) := \{ u \in \shH^r_{X}(\Omod_U)\ |\ \shI_X^{k+1} u = 0 \}, \quad k \geq 0,
\]
where \(\shI_X\) is the ideal sheaf defining \(X\) on \(U\). The order filtration is compatible with the filtration by order of differential operators on \(\Dmod_U\),
\[
F_l \Dmod_U \cdot O_k \shH^r_{X}(\Omod_X) \subseteq O_{k+l} \shH^r_{X}(\Omod_U), \quad \textnormal{for}\quad k, l \geq 0.
\]
Since \(r = \codim_U(X)\), the sheaves \(O_k \shH_{X}^r(\Omod_X)\) are coherent. In general, this is no longer true for arbitrary \(r\) in the non-local complete intersection case.

\jump

Recall the definition of local cohomology in terms of the Ext functor, see \cite[2.8]{Har67},
\[
\shH^r_{X} (\Omod_U) = \lim_{\longrightarrow} \shExt^r_{\Omod_U}(\Omod_U / \shI_X^k, \Omod_X).
\]
For \(r = \codim_U(X)\), one has that the Ext filtration is
\[
E_k \shH_{X}^r(\Omod_U) := \im[\shExt^r_{\Omod_U}(\Omod_U/\shI_X^{k+1}, \Omod_X) \longrightarrow \shH^r_{X}(\Omod_U)] = \shExt^r_{\Omod_U}(\Omod_U/\shI_X^{k+1}, \Omod_U),
\]
see \cite[\S 7]{MP22}. The sheaves \(E_k \shH^r_{X}(\Omod_U)\) are also coherent, but it is not clear a priori that they are compatible with the filtration on \(\Dmod_U\). However, we have that for \(r = \codim_U(X)\), 
\[O_k \shH^r_{X}(\Omod_U) = E_k \shH^r_{X}(\Omod_U),\quad \textnormal{for all}\quad k \geq 0,\]
see \cite[7.14]{MP22}.

\jump

After the above discussion, we define the following filtration on the complex \(\shK^\bullet \) 
\begin{equation} \label{eq:F}
    F_{k} \shK^p = \bigoplus_{i + j \leq k + d + p} O_i \shH^r_{X}(\Omod_U) \otimes_{\Omod_U} \Omega^{d+p}_U \otimes_{\CC} \partial_t^j, \qquad \textnormal{for} \qquad p = -d, \dots, 0.
\end{equation}
The filtration \(F_\bullet\) is then an increasing filtration by coherent \(\Omod_U\)-submodules such that \(\partial_t\) maps \(F_k\) to \(F_{k+1}\). The filtration \eqref{eq:F} induces a filtration on the Gauss-Manin system by
\[
    F_{k} \GM = \im\,[ \derR f_*(F_{k} \shK^\bullet) \longrightarrow \GM].
\]

The local cohomology sheaf \(\shH_{X}^r(\Omod_U)\) is a regular holonomic \(\Dmod_U\)-module underlying a mixed Hodge module which, in particular, implies that \(\shH_{X}^r(\Omod_U)\) carries a canonical Hodge filtration \(F^H_\bullet\). A Hodge filtration can then be defined on the complex \(\shK^\bullet\) by
\[
    F^H_{k} \shK^p = \bigoplus_{i + j \leq k + d + p} F^H_i \shH^r_{X}(\Omod_U) \otimes_{\Omod_U} \Omega^{d+p}_U \otimes_{\CC} \partial_t^j, \qquad \textnormal{for} \qquad p = -d, \dots, 0,
\]
and similarly \(F_k^H \shG_f^\bullet\) on \(\shG_f^\bullet\). In general, we only have the inclusions
\begin{equation} \label{eq:inclusion}
    F^H_k \shH^r_{X}(\Omod_U) \subseteq O_k \shH^r_{X}(\Omod_U) = E_k \shH^r_{X}(\Omod_U), \quad k \geq 0,
\end{equation}
see \cite[7.4]{MP22}, and the equality holds for all \(k \geq 0\) if and only if \(X\) is smooth, \cite[9.6]{MP22}. Therefore, the Hodge filtration \(F^H_\bullet\) and the filtration \(F_\bullet\) from \eqref{eq:F} on \(\shK^\bullet\) are generally not equal. When \(X\) is smooth both filtrations are locally given by
\begin{equation} \label{eq:smooth-case}
    F_k \shH^r_X (\Omod_U) = F_k^H \shH^r_X (\Omod_U) \simeq \bigoplus_{\alpha_1 + \cdots + \alpha_r \leq k} \Omod_X \otimes \partial_{x_1}^{\alpha_1} \cdots \partial_{x_r}^{\alpha_r},
\end{equation}
where \(x_1, \dots, x_r\) are local coordinates on \(U\) defining \(X\).

\subsection{Regular differential forms} \label{sec:regular-forms}

We will review next some basic facts about the complex \(\omega^\bullet_X\) of regular differential forms of a pure dimensional analytic space introduced by Barlet in \cite{Bar78}, but see also the work of Kersken \cite{Ker84}. For simplicity, we will restrict the exposition to the case of a pure dimensional local complete intersection.

\jump

Let \(X\) be a local complete intersection of pure dimension \(n\) contained in a complex manifold \(U\) of dimension \(n + r\). Recall that the fundamental class \(c_{X}\) of \(X\) is an element of \(H_X^r(U, \Omega_U^r)\) that is defined as follows. If \(F_1, \dots, F_r\) define \(X\) in \(U\), the fundamental class \(c_{X}\) corresponds to the class of the element
\[
    \frac{\dd F_1}{F_1} \wedge \cdots \wedge \frac{\dd F_r}{F_r},
\]
under the isomorphism \(H^r_X(U, \Omega^r_U) \simeq H^0(U, \shH^r_{X}(\Omega^r_U))\). In particular, \(c_{X}\) is a global section of the sheaf \(\shExt^r_{\Omod_U}(\Omod_X, \Omega^r_U)\). The morphism of complexes \( \Omega^\bullet_U \longrightarrow \shExt^r_{\Omod_U}(\Omod_X, \Omega_U^{r + \bullet}) \) defined by exterior multiplication by \(c_{X}\) induces the fundamental class morphism
\[
    \bar{c}_{X} : \Omega^\bullet_{X} \longrightarrow \shExt^r_{\Omod_U}(\Omod_X, \Omega_U^{r + \bullet}),
\]
which is intrinsic and commutes with the exterior derivatives, see \cite[1, 2]{Bar78}.

\jump

Let \(\shS_X\) be the singular locus of \(X\) and denote by \(j : X \setminus \shS_X \longrightarrow X\) the open embedding. The functoriality of the exact sequence
\[
0 \longrightarrow \shH_{\shS_X}^0(\shM) \longrightarrow \shM \longrightarrow j_* j^{-1} \shM \xrightarrow{\, \ \partial \ \, } \shH_{\shS_X}^1(\shM) \longrightarrow 0,
\]
for any sheaf \(\shM\) of \(\Omod_U\)-modules, induces a morphism
\[
    \shH^1_{\shS_X}(\bar{c}_{X}) : \shH^1_{\shS_X}(\Omega^\bullet_X) \xrightarrow{\qquad} \shH^1_{\shS_X}(\shExt^r_{\Omod_U}(\Omod_X, \Omega^{r+\bullet}_U)).
\]
Then the sheaves \(\omega_X^\bullet\) are defined as the kernels of the homomorphism
\[
\partial \circ \shH^1_{\shS_X}(\bar{c}_{X}) : j_* j^{-1} \Omega^\bullet_X \xrightarrow{\qquad} \shH^1_{\shS_X}(\shExt^r_{\Omod_U}(\Omod_X, \Omega^{r+\bullet}_U)).
\]
Since \(j_*j^{-1} \Omega^\bullet_X \) has no torsion, the sheaves \(\omega_X^p \) are torsion-free \(\Omod_X\)-modules which coincide with \(\Omega^p_X\) on the smooth points of \(X\). In particular, \(\omega^p_X = 0\) for \(p < 0\) and \(p > n\). Moreover, the exterior derivative induces a complex \((\omega_X^\bullet, \dd)\). 

\jump

Since the complex \(\shM^\bullet = \shExt_{\Omod_U}^{r}(\Omod_X, \Omega^{r + \bullet}_U)\) has no \(\Omod_X\)-torsion, one can form the commutative diagram
\begin{equation} \label{eq:diagram-dualizing}
    \begin{tikzcd}
        0 \arrow{r} & \omega^\bullet_X \arrow[dashed,d] \arrow{r} & j_*j^{-1} \Omega^\bullet_X \arrow{d} \arrow{dr} \arrow{r}{\partial} & \shH^1_{\shS_X}(\Omega^\bullet_X) \arrow{d}{\shH_{\shS_X}^1(\bar{c}_{X})} \arrow{r} & 0 \\
        0 \arrow{r} & \shM^\bullet \arrow{r} & j_* j^{-1} \shM^\bullet \arrow{r} & \shH^1_{\shS_X}(\shM^\bullet) \arrow{r} & 0.
    \end{tikzcd}
\end{equation}
From this diagram, one finds that there is a unique morphism \(\omega^\bullet_X \longrightarrow \shExt^r_{\Omod_U}(\Omod_X, \Omega^{r+\bullet}_U)\), which we keep denoting by \(\bar{c}_{X}\). This morphism sits in the exact sequence
\begin{equation} \label{eq:exact-sequence}
0 \longrightarrow \omega^\bullet_X \xrightarrow{\ \bar{c}_{X}\ } \shExt^r_{\Omod_U}(\Omod_X, \Omega_U^{r + \bullet}) \xrightarrow{\ \ \alpha \ \ } [ \shExt^r_{\Omod_U}(\Omod_X, \Omega^{r + 1 + \bullet}_U) ]^r,
\end{equation}
where \(\alpha(\eta) = (\eta \wedge \dd F_1, \dots, \eta \wedge \dd F_r)\), see \cite[4]{Bar78}. In particular, \(\omega^n_X\) is isomorphic to \(\shExt^r_{\Omod_U}(\Omod_X, \Omega_U^{r + n})\), the Grothendieck dualizing sheaf. Since \(X\) is a local complete intersection, the sheaf \(\omega^n_X\) is a locally free of rank one generated by the Gelfand-Leray form of \(X\), that is, locally \(\omega^n_X \simeq \Omod_X (\dd z / \dd F_1 \wedge \cdots \dd F_r) \).

\jump

As shown in \cite[3]{Bar78}, the sheaves of regular differential forms satisfy the following duality property
\begin{equation} \label{eq:duality}
\omega_X^p \simeq \shHom_{\Omod_X}(\Omega_X^{n-p}, \omega_X^n).
\end{equation}
Since \(\omega_X^n\) is coherent, the duality implies that \(\omega^p_X, p = 0, \dots, n - 1,\) are also coherent. A posteriori, one can define the regular differential forms \(\omega_X^p, p \in [0, n),\) as those meromorphic differential forms \(\omega\) on \(X\) of degree \(p\) such that \(\omega \wedge \eta \in \omega^n_X\), for any \(\eta \in \Omega_X^{n-p}\).

\subsection{The relative case} \label{sec:relative-forms}

Assume now that \(f : X \longrightarrow S\) is a morphism of analytic spaces from a locally complete intersection \(X\) to an analytic space \(S\). Then the definitions from section \ref{sec:regular-forms} extend straightforwardly to the relative case. For simplicity, we will assume that \(S\) is pure dimensional of dimension one.

\jump

Assume that \(X\) has pure dimension \(n + 1\) and that it is embedded inside a complex manifold \(U\) of pure dimension \(n + r + 1\). Taking \(F : U \longrightarrow S\) such that \(f = \restr{F}{X}\), one defines the fundamental class \(c_f\) of the morphism \( f \) as the class of 
\[
    \frac{\dd F_1}{F_1} \wedge \cdots \wedge \frac{\dd F_r}{F_r} \wedge \dd F
\]
in \(H^0(U, \shH_X^r(\Omega_U^{r+1}))\), where \(F_1, \dots, F_r\) define \(X\) in \(U\). Moreover, \(c_{f}\) induces a fundamental class morphism
\[
    \bar{c}_{f} : \Omega_f^\bullet \longrightarrow \shExt_{\Omod_U}^{r}(\Omod_X, \Omega_U^{r + 1 + \bullet}),
\]
where \(\Omega_f^\bullet\) is the complex of relative K\"ahler differentials of the morphism \(f : X \longrightarrow S\). It is easy to check that neither the fundamental class nor the induced fundamental morphism depends on the choice of \(F\) defining \(f\). 

\jump

Let \(\shC_f\) be the critical locus of \(f\), and denote by \(j : X \setminus \shC_f \longrightarrow X\) the open embedding. As in \cref{sec:regular-forms}, define the sheaves of relative regular forms \(\omega_f^\bullet\) as the kernels of the homomorphisms
\[
\partial \circ \shH^1_{\shC_f}(\bar{c}_{f}) : j_* j^{-1} \Omega_f^\bullet \longrightarrow \shH_{\shC_f}^1(\shExt_{\Omod_U}^1(\Omod_X, \Omega^{r+1+\bullet}_U)).
\]
These give rise to torsion-free \(\Omod_X\)-modules that form a complex \((\omega_f^\bullet, \dd)\) of \(f^{-1}\Omod_S\)-modules coinciding with \(\Omega^\bullet_f\) on the regular points of \(f\).  Notice, we have the vanishing \(\omega_f^p = 0\) for \(p < 0\) and \(p > n\). The same proof as in \cite[4]{Bar78} shows that the relative regular forms sit in the exact sequence
\begin{equation} \label{eq:relative-exact-sequence}
0 \longrightarrow \omega^\bullet_f \xrightarrow{\ \bar{c}_{f}\ } \shExt^r_{\Omod_U}(\Omod_X, \Omega_U^{r + 1 + \bullet}) \xrightarrow{\ \ \beta \ \ } [ \shExt^r_{\Omod_U}(\Omod_X, \Omega^{r + 2 + \bullet}_U) ]^{r+1},
\end{equation}
where \(\beta(\eta) = (\eta \wedge \dd F_1, \dots, \eta \wedge \dd F_r, \eta \wedge \dd F)\). From this, one can deduce the inclusions \(\omega_{f}^p \subseteq \omega_X^{p+1}\) and that \(\omega^n_f \) is isomorphic to the Grothendieck dualizing sheaf \(\omega^{n+1}_X\). The duality satisfied in this case is 
\begin{equation} \label{eq:duality-relative}
\omega_f^p \simeq \shHom_{\Omod_X}(\Omega^{n-p}_f, \omega_f^n) \simeq \shHom_{\Omod_X}(f^* \Omega^{1}_S, \omega_{X}^{n+1}),
\end{equation}
as one can verify by using the same arguments in the proof of \cite[3]{Bar78}. We can deduce from this that the sheaves \(\omega^p_f, p = 0, \dots, n - 1\), are coherent. Since \(X\) is a local complete intersection, \(\omega_X^{n+1}\) is locally free, and it follows from \eqref{eq:duality-relative} that \(\omega^n_{f}\) coincides with the relative dualizing sheaf \(\omega^{n+1}_X \otimes (f^* \omega^1_S)^*\), see for instance \cite{Del83}. Particularly, the sheaf \(\omega_{f}^n\) is compatible with base change. However, the complex \(\omega^\bullet_f\) is not, see \cref{rmk:base-change}.

\subsection{Normal varieties} \label{sec:normal}

We keep using the notations from the previous sections, but assume now that \(X\) is a normal analytic space. Moreover, assume that the critical locus of \(f : X \longrightarrow S\) is contained in the singular locus of \(X\).

\jump

For a coherent sheaf \(\shF\) on \(X\), we denote its \(\Omod_X\)-dual by \(\shF^* := \shHom_{\Omod_X}(\shF, \Omod_X)\). The double dual, or reflexive hull, is denoted by \(\shF^{**} := (\shF^*)^*\). There is always a canonical map \(\shF \longrightarrow \shF^{**}\), and the sheaf \(\shF\) is called reflexive if this map is an isomorphism. This map gives rise to the four-term exact sequence
\[
0 \longrightarrow \textnormal{tor}\,\shF \longrightarrow \shF \longrightarrow \shF^{**} \longrightarrow \textnormal{cotor}\,\shF \longrightarrow 0,
\]
where \(\textnormal{tor}\,\shF\) is the torsion subsheaf of \(\shF\), and where \(\textnormal{cotor}\,\shF\), the cotorsion sheaf of \(\shF\), is the cokernel of the map \(\shF \longrightarrow \shF^{**}\).

\jump

On a normal analytic space \(X\), the dualizing sheaf \(\omega_X^{n+1}\) is reflexive \cite[5.3]{GPR94}, and hence so is \(\omega_f^n\). From this fact and the dualities in \eqref{eq:duality} and \eqref{eq:duality-relative}, one deduces the reflexivity of \(\omega_X^p\), for \(p = 0, \dots, n\), and of \(\omega_f^p\), for \(p = 0, \dots, n - 1\), see for instance \cite[\href{https://stacks.math.columbia.edu/tag/0AV3}{0AV3}]{stacks-project}. By \cite[7]{Ser66}, reflexivity on a normal analytic space implies that 
\begin{equation} \label{eq:reflexive-differential}
    \omega_X^p \simeq j_* j^{-1} \Omega_X^p \simeq \Omega^{[p]}_X, \qquad \omega_f^p \simeq j_* j^{-1} \Omega_f^p \simeq \Omega^{[p]}_f,
\end{equation}
where \(\Omega^{[p]}_X := (\Omega^p_X)^{**}, \Omega^{[p]}_f := (\Omega_f^p)^{**}\) are the (relative) reflexive differential forms on \(X\). This makes the (relative) regular differential easier to handle in the normal case, as proving assertions for \(\omega^\bullet_X\) or \(\omega^\bullet_f\) is equivalent to proving assertions on the regular locus. Another consequence is that the sheaves of regular differential forms sit in the exact sequence
\begin{equation} \label{eq:four-term-exact}
0 \longrightarrow \textnormal{tor}\, \Omega^\bullet_X \longrightarrow \Omega^\bullet_X \longrightarrow \omega_X^\bullet \longrightarrow \textnormal{cotor}\, \Omega^\bullet_X \longrightarrow 0,
\end{equation}
and similarly for \(\omega^\bullet_f\). For \(p \leq n\), one has \(\textnormal{tor}\, \Omega^p_X \neq 0\) if and only if \(p = n + 1\), and \(\textnormal{cotor}\, \Omega^p_X \neq 0 \) for \(p = n, n+1\), see for instance \cite[3.1]{MV19}. The relative case will be treated in \cref{lemma:torsion}. Notice that the vanishing \(\omega_X^p = 0\), for \(p > n + 1\), implies that \(\Omega^p_X\) is torsion for \(p > n + 1\). Likewise, \(\Omega_f^p \) is torsion for \(p > n\). 

\jump

Passing to the stalks at \(x \in X\), the cotorsion of \(\Omega^{n+1}_{X,x}\) is isomorphic to the singular locus \(\Omod_{\shS_X, x}\) of \((X, x)\). Its length coincides with \(\tau'(X, x) = \dim_{\CC} \textnormal{tor}\, \Omega^{n+1}_{X,x}\), see \cite[1]{Gre80}. As shown in \cite[4.6]{Ker84}, the length of \(\textnormal{cotor}\, \Omega^{n}_{X,x}\) is equal to \(\tau(X, x)\) the Tjurina number of \((Z, x)\). Notice that in the hypersurface case, \(\tau(X, x) = \tau'(X, x)\).

\jump

Define the torsion-free differentials 
\[
\check{\Omega}^p_X := \Omega^p_X / \textnormal{tor}\, \Omega^p_X, \qquad \check{\Omega}^p_f := \Omega^p_f / \textnormal{tor}\, \Omega_f^p,
\]
so that the exact sequence \eqref{eq:four-term-exact} reduces to the short exact sequence,
\begin{equation} \label{eq:short-exact-seq}
0 \longrightarrow \check{\Omega}^\bullet_X \longrightarrow \omega_X^\bullet \longrightarrow \textnormal{cotor}\, \Omega^\bullet_X \longrightarrow 0,
\end{equation}
and similarly in the relative case. Torsion-free differential forms are not, in general, reflexive, but they are functorial, see \cite[1.1]{Fer70}.

\jump

In the absolute case, all this discussion can be summarized by the maps
\[
\textnormal{tor}\, \Omega^p_X \xhookrightarrow{\quad} \Omega^p_X \xrightarrowdbl{\quad} \check{\Omega}^p_X \xhookrightarrow{\quad} \omega^p_X.
\]
The relation between all these notions of differential forms on a normal space can be described as follows. A torsion differential form is a K\"ahler form that vanishes outside the singular locus of \(X\). Giving a regular differential form is the same as giving a K\"ahler differential form on the smooth locus of \(X\). A regular, i.e., reflexive, differential is a torsion-free differential if it extends to the ambient space of an embedding of \(X\) into a smooth space, see \cite{Fer70}.

\subsection{Regular forms on ICIS} \label{sec:regular-icis}

We are interested in the case where \(f : X \longrightarrow S\) is a good representative of a smoothing \(f : (X, x) \longrightarrow (\CC, 0)\) of an ICIS \((Z, x)\) of dimension \(n\). Recall that, in this case, both the singular locus \(\shS_X\) of \(X\) and the critical locus \(\shC_f\) of \(f\) are concentrated on \(x \in X\). For the rest of this work, we will need to assume that \((Z, x)\) is positive-dimensional, i.e., \(n > 0\). Since \((X, x)\) is a complete intersection, it is Cohen-Macaulay, so by Serre's criterion, this will be equivalent to \((X, x)\) being a normal singularity.

\jump

In this setting, we will study the cohomology of the complexes of (relative) regular differential forms, as they will play an important role in the following sections. Since the singularities are isolated, it is enough to study the cohomology of the stalk at \(x \in X\). In the absolute case, Kersken \cite[4.6]{Ker84} showed that, if \((X, x)\) is an ICIS of dimension \(n + 1, n > 0\), then 
\begin{equation} \label{eq:vanishing}
H^p(\omega_{X,x}^\bullet) = 0, \quad \textnormal{for} \quad p = 1, \dots, n - 1,
\end{equation} 
see also \cite[\S 4]{Gre75}. Moreover, 
\begin{equation} \label{eq:vanishing2}
    \dim_{\CC} H^0(\omega_{X,x}^\bullet) = 1, \quad \dim_{\CC} H^{n+1}(\omega_{X,x}^\bullet) - \dim_{\CC} H^{n}(\omega_{X,x}^\bullet) = \mu(X, x) - \tau(X, x).
\end{equation}
It is known that \(\mu(X, x) \geq \tau(X, x)\), see  \cite{LS85}, hence \(\dim_{\CC} H^{n+1}(\omega^\bullet_{X,x}) \geq \dim_{\CC} H^n(\omega^\bullet_{X,x})\).

\jump

We will require a similar vanishing result in the relative case. Consider the relative version of the short exact sequence of complexes from \eqref{eq:short-exact-seq},
\begin{equation} \label{eq:short-exact-regular}
0 \longrightarrow \check{\Omega}^\bullet_{f,x} \longrightarrow \omega_{f,x}^\bullet \longrightarrow \textnormal{cotor}\, \Omega^\bullet_{f,x} \longrightarrow 0.
\end{equation}
The long exact sequence in cohomology will yield the required vanishings as long as one understands when the torsion and cotorsion modules of \(\Omega^\bullet_f\) vanish. For that, one can employ the generalized Koszul complexes from \cite{MV19}.

\begin{lemma} \label{lemma:torsion}
    Let \(f : (X, x) \longrightarrow (S, s)\) be a flat morphism of germs of analytic spaces with \((S, s)\) smooth of dimension one. Let \(x \in X\) be a critical point of \(f\) such that the fiber \((X_s, x)\) is a normal complete intersection singularity of dimension \(n\). Then
    \[
        \begin{split}
        \textnormal{tor}\, \Omega^p_{f,x} = 0,& \quad \textnormal{for} \quad p < k, \\
        \textnormal{cotor}\, \Omega^p_{f,x} = 0,& \quad \textnormal{for} \quad p < k - 1,
        \end{split}
    \]
    where \(k\) is the codimension of the critical locus of \(f\) at \(x\).
\end{lemma}
\begin{proof}
    We can always choose a representative \(f : X \longrightarrow S\) that is flat at every point of \(X\). Moreover, the assumptions imply that \((X, x)\) is also a complete intersection singularity \cite[6.10]{Loo84} that is normal \cite[1.101]{GLS07}. After embedding \(X\) in a smooth ambient space \(U\), we can consider the morphism \( F = (F_1, \dots, F_r, F_{r+1}) : U \longrightarrow \CC^{r+1}\), where \(F_1, \dots, F_r\) generate \(X\) on \(U\), and \(F_{r+1}\) is such that \( \restr{F_{r+1}}{X} = f\). Consider the relative conormal sequence for the morphism \(F : X \longrightarrow \CC^{r+1}\),
    \[
        F^* \Omega^1_{\CC^{r+1}} \longrightarrow \Omega_U^1 \longrightarrow \Omega_F^1 \longrightarrow 0.
    \]
    Since \(X\) is a complete intersection, it follows from base change that \(\Omega^1_F \otimes \Omod_X \simeq \Omega^1_f \). Thus, tensoring the previous exact sequence by \(\Omod_X\) one gets,
    \begin{equation} \label{eq:short-exact-relative}
        F^* \Omega^1_{\CC^{r+1}} \otimes_{\Omod_U} \Omod_X \longrightarrow \Omega_U^1 \otimes_{\Omod_U} \Omod_X \longrightarrow \Omega^1_f \longrightarrow 0.
    \end{equation}
    The module \(F^* \Omega^{1}_{\CC^{r+1}}\) is locally free of rank \(r + 1\), and denote by \(d\) the rank of \(\Omega^1_U\). For \(y \in X\) a regular point of the morphism \(f\), the relative differentials \(\Omega_{f}^1\) are free of rank the dimension \(n \) of the fiber \(f^{-1}(y)\). The flatness of \(f\) implies that the fiber dimension is constant, thus \(d = n + r + 1\). Therefore, localizing at \(y \in X\), \eqref{eq:short-exact-relative} is exact on the left. Since \(X\) is normal at \(x \in X\), \(X\) is reduced at \(x\), and \(X\) is generically a submanifold. Hence, \eqref{eq:short-exact-relative} is generically exact on the left, and since \(F^* \Omega^1_{\CC^{r+1}}\) is locally free, \eqref{eq:short-exact-relative} is left exact everywhere. This yields,
    \begin{equation} \label{eq:free-res}
        0 \longrightarrow \Omod_{X,x}^{r+1} \xrightarrow{\ \varphi \ } \Omod_{X, x}^{n + r + 1}\longrightarrow \Omega^1_{f,x} \longrightarrow 0,
    \end{equation}
    a finite free resolution of \(\Omega^1_{f}\) as a \(\Omod_{X,x}\)-module. Fixing local coordinates, \(\varphi = (a_{i,j})\) is just the restriction of the Jacobian matrix given by \(\dd F_j = \sum_{i} a_{i,j} \dd z_i, \) for \(j = 1, \dots, r+1\).

    \jump
    
    Consider the generalized Koszul complex \(\mathscr{D}_p = \mathscr{D}_p({\varphi})\) from \cite[\S 1.4]{MV19} associated to the injective map \(\varphi\) in \eqref{eq:free-res}. It follows from \cite[3.5]{MV19} applied to the \(\Omod_{X, x}\)-module \(\Omega^1_{f,x}\) and the normality of \((X, x)\), that
    \[
        H_i(\mathscr{D}_p) \simeq \left \{
        \begin{matrix*}[l] 
            \textnormal{tor}\, \Omega^p_f, \qquad i = n - p + 1, \\
            \textnormal{cotor}\, \Omega^p_f, \quad i = n - p.
        \end{matrix*}
        \right .
    \]
    The lemma then follows from using \cite[2.5]{MV19}. One has that \(H_i(\mathscr{D}_p) = 0\), if and only if \(i > n + 1 - k\), where \(k = \textnormal{codim}_x \shC_f\). 
\end{proof}

For a smoothing of an ICIS, the codimension of the singular locus is \(n + 1\), so the previous lemma and \eqref{eq:short-exact-regular} imply that \(\Omega^p_f\) is torsion-free for \(p < n+1\), and torsion otherwise. The cotorsion module is non-zero if and only if \(p = n\). In particular, \(\Omega^p_{f,x} \simeq \omega^p_{f,x}\), for \(p < n\). The module \(\textnormal{cotor}\, \Omega_{f,x}^{n}\) coincides with the critical locus \(\Omod_{\shC_f, x}\) and its length equals the multiplicity \(\nu(D_f, 0)\) of the discriminant of the smoothing \(f\), see \cite[p. 63]{Loo84}. 

\begin{proposition} \label{prop:cohomology-relative}
    Let \( f : (X, x) \longrightarrow (\CC, 0)\) be a smoothing of an ICIS \((Z, x)\) of dimension \(n > 0\). Then 
    \[
        H^p(\omega^\bullet_{f,x}) = 0, \quad \textnormal{for} \quad 0 < p < n. 
    \]
    Furthermore, 
    \[H^{n}(\omega^\bullet_{f,x}) \simeq \frac{\omega^{n}_{f,x}}{\dd \Omega^{n-1}_{f, x}}\]
    is a free \(\Omod_{\CC, 0}\)-module of rank \(\mu(Z, x)\).
\end{proposition}
\begin{proof}
Considering the long exact sequence in cohomology of \eqref{eq:short-exact-regular}, 
\[
\cdots \longrightarrow H^p(\check{\Omega}_{f, x}^\bullet) \longrightarrow H^p(\omega_{f,x}^\bullet) \longrightarrow H^p(\textnormal{cotor}\, \Omega^\bullet_{f,x}) \longrightarrow H^{p+1}(\check{\Omega}^\bullet_{f,x}) \longrightarrow \cdots,
\]
\cref{lemma:torsion} implies \(H^p(\check{\Omega}_{f,x}^\bullet) = H^p(\Omega^\bullet_{f,x})\) for \(p = 0, \dots, n - 1,\) and \(H^p(\textnormal{cotor}\, \Omega^\bullet_{f,x}) = 0,\) for \(p = 0, \dots, n - 1\). Moreover, \(H^p(\Omega_{f, x}^\bullet) = 0, \) for \(p = 1, \dots, n - 1\), see \cite[8.20]{Loo84}. This gives the first part of the result. For the second part, since \(f\) is a smoothing \(\omega^n_{f,x}/\dd \Omega_{f,x}^{n-1}\) is a free \(\Omod_{\CC, 0}\)-module of rank \(\mu(Z, x)\), see \cite[8.8]{Loo84}. 
\end{proof}

In particular, if \((Z, 0)\) is a hypersurface singularity and \(f : (\CC^{n+1}, x) \longrightarrow (\CC, 0)\) is a defining germ, we can identify \(H^n(\omega_{f,x}^\bullet)\) with the Brieskorn lattice \(H''_{f,x}\) from \cite[\S 2.4]{Bri70} using the isomorphism \eqref{eq:exact-sequence}. We will keep this notation and denote \(H^n(\omega^\bullet_{f,x})\) by \(H_{f,x}''\). 

Another consequence of the proof of \cref{prop:cohomology-relative} is that, since \(\Omega^{n+1}_{f,x}\) is torsion, the long exact sequence in cohomology of \eqref{eq:short-exact-regular} reduces to the short exact sequence of \(\Omod_{\CC, 0}\)-modules
\begin{equation} \label{eq:short-exact-lattice}
0 \longrightarrow \frac{\Omega^n_{f,x}}{\dd \Omega^{n-1}_{f,x}} \longrightarrow \frac{\omega_{f,x}^{n}}{\dd \Omega^{n-1}_{f,x}} \longrightarrow \textnormal{cotor}\, \Omega^n_{f,x} \longrightarrow 0.
\end{equation}
The first module in \eqref{eq:short-exact-lattice} is usually denoted by \(H'_{f,x}\).

\jump

\begin{remark} \label{rmk:base-change}
Even if \(\omega^n_{f}\) is compatible with base change, the complex \(\omega_f^\bullet\) is not. Indeed, assuming compatibility with base change, the fact that the complex is free of \(\Omod_X\)-torsion yields the short exact sequence
\[
0 \longrightarrow \omega^{\bullet}_{f, x} \xrightarrow{\ t\cdot\ } \omega^{\bullet}_{f,x} \longrightarrow \omega^\bullet_{Z, x} \longrightarrow 0.
\]
By \eqref{eq:vanishing}, the long exact sequence in cohomology gives
\begin{equation} \label{eq:torsion-sequence}
0 \longrightarrow H^{n-1}(\omega^\bullet_{Z, x}) \longrightarrow H^n(\omega^\bullet_{f,x}) \xrightarrow{\ t\cdot\ } H^n(\omega^{\bullet}_{f,x}) \longrightarrow H^n(\omega^\bullet_{Z,x}) \longrightarrow 0.
\end{equation}
On the other hand, since \(H^n(\omega^\bullet_{f,x})\) is free of rank \(\mu(Z, x)\) by \cref{prop:cohomology-relative}, base change implies
\[
\dim_{\CC} H^n(\omega^\bullet_{Z, x}) = \dim_{\CC} \frac{\omega^n_{Z,x}}{\dd \omega^{n-1}_{Z, x}} = \mu(Z, x).
\]
Thus, by \eqref{eq:vanishing2} one has \(\dim_{\CC} H^{n-1}(\omega^\bullet_{Z,x}) = \tau(Z, x) > 0\). Therefore, \eqref{eq:torsion-sequence} implies that \(H^n(\omega^\bullet_{f,x})\) has non-zero \(\Omod_{\CC,0}\)-torsion, contradicting \cref{prop:cohomology-relative}.
\end{remark}

\subsection{The cohomology sheaves of \texorpdfstring{\(\shK^\bullet\)}{K}} \label{sec:cohomology}

In this section, we will compute the cohomology sheaves of the complex \(\shK^\bullet\). These sheaves are naturally \(f^{-1} \Dmod_S\)-modules with a good filtration induced by the filtration \(F\) on \(\shK^\bullet\). First, we need to introduce some technical results.

\jump

Let \(X\) be a reduced \(n\)-dimensional complex analytic space, not necessarily normal, embedded in a complex manifold \(U\) of dimension \(n + r\). Let \(\shF\) be a locally free sheaf of \(\Omod_U\)-modules of finite rank. The sheaves \(\widehat{\shF}_X = \shExt_{\Omod_U}^r(\Omod_X, \shF)\) are coherent \(\Omod_X\)-modules that satisfy the following Riemann extension property.

\begin{lemma} \label{lemma:extension}
Let \(X\) be a reduced complex analytic space, \(U \subseteq X\) an open subset, and \(A \subset U \) an analytic subset of codimension \(\leq 2\). Then the restriction map
\[
    \widehat{\shF}_X(U) \longrightarrow \widehat{\shF}_X(U \setminus A)
\]
is a bijection.
\end{lemma}

Sheaves satisfying this property are also called normal sheaves. Moreover, if we assume \(X\) to be normal, then \(\widehat{\shF}_X\) is a reflexive sheaf, see \cite[1.6]{Har80}.

\jump

This property is probably well-known for the dualizing sheaf \(\omega^n_{X} \simeq \shExt^r_{\Omod_U}(\Omod_X, \Omega_{U}^{r+n})\) and the proof \cref{lemma:extension} is the same as in this case, see \cite[3.1]{GR70}, since only the fact that \(\Omega^{r+n}_U\) is locally free is used. If \(X\) is a local complete intersection, one can use the Koszul resolution given by a local regular sequence to calculate the \(\shExt\) sheaf, yielding
\[
\widehat{\shF}_X \simeq \shF \otimes_{\Omod_X} \textnormal{det}\, \shN_{X|U},
\] 
where \(\shN_{X|U}\) is the normal bundle of \(X\) on \(U\). As a consequence, \(\widehat{\shF}_X\) is also locally free of finite rank. Hence, in the local complete intersection case, \(\widehat{\shF}_X\) is automatically reflexive and, if \(X\) is normal, \cref{lemma:extension} follows from \cite[1.6]{Har80}.

\jump

We will only be interested in the case where \(\shF = \Omega^{p}_{U}, p = 0, \dots, n + r\). As a consequence of \cref{lemma:extension}, when \(X\) is a normal space, we have the following isomorphism of sheaves
\[
    \shExt_{\Omod_U}^r(\Omod_X, \Omega^{r + p}_U) \simeq j_* j^{-1} \shExt_{\Omod_U}^r(\Omod_X, \Omega^{r + p}_U),
\]
for \(p = 0, \dots, n + r\), where \(j : X \setminus \shS_X \longrightarrow X\) is the open embedding of the complement of the singular locus.

\jump

Recall from \cref{sec:graph} that the complex \(\shK^\bullet\) is connected with the complex \(\shH_X^r(\Omega_U^{d+\bullet})\). The complex \(\shH^r_{X}(\Omega^{d + \bullet}_U) \) placed in degrees \(-d, \dots, 0\) coincides with \(\DR_U\shH_X^r(\Omod_U)\), and it follows from \cite[1.1]{Meb76} that it is a resolution of the constant sheaf \(i_* \shCC_X\), where \(i : X \xhookrightarrow{\quad} U \) is the closed embedding. In other words, the Poincar\'e lemma holds for the differential forms \(\shH_X^r(\Omega_U^{d+\bullet})\). 

\begin{proposition} \label{prop:cohomology-derham}
    Let \( f : X \longrightarrow S\) be a good representative of a smoothing of an ICIS of dimension \(n > 0\). We have the following isomorphism of \(f^{-1}\Omod_{S}\)-modules
    \[
        F_k \shH^p ( \shK^\bullet ) \simeq \bigoplus_{l \leq k} \shH^{n+p}(\omega^\bullet_{f}) \otimes \partial_t^l, \qquad p = -d, \dots, 0.
    \]
    In particular, \(\shH^p(\shK^\bullet) = 0 \), for \(p \neq -n \) or \( 0\), and \(\shH^0(\shK^\bullet)\) is concentrated on \(x \in X\).
\end{proposition}
\begin{proof}
Assume \(X\) is embedded in a complex manifold \(U\) and choose \(F\) such that \(f = \restr{F}{X}\). Working locally around a regular point \(x \in X\) of \(f\), let \(x_1, \dots, x_{r+1}\) be local coordinates in \(U\) such that \(F_1 = x_1, \dots, F_r = x_r\) define \(X\) and \(F = x_{r+1} = t\).

\jump

We will focus first on the piece of degree zero. Giving a section of degree zero \(\omega\) of \(\shK^p\) amounts to give a section of \(\shExt^r_{\Omod_U}(\Omod_X, \Omega^{d+p}_U) \). Around a regular point, the image of such a section by the differential of the complex \(\shK^\bullet\) can be expressed as, 
\[
    \underline{\dd}(\omega) = \underline{\dd}\left(\left[\frac{\eta}{x_1 \cdots x_r}\right]\right) = \left[\frac{\dd \eta}{x_1 \cdots x_r}\right] + \sum_{i = 1}^{r+1} \left[ \frac{\dd x_i \wedge \eta }{x_1 \cdots x_r} \right] \otimes \partial_{x_i},
\]
for some \(\eta\) a section of \(\Omega^{d+p}_{U}\), see \eqref{eq:smooth-case}. The cocycle condition implies that 
\begin{equation} \label{eq:prop-cohomology}
    \left[\frac{\dd \eta}{x_1 \cdots x_r}\right] = 0, \qquad \left[\frac{\dd x_i \wedge \eta}{x_1 \cdots x_r}\right] = 0, \quad \textnormal{for} \quad i = 1, \dots, r+1.
\end{equation}
It follows from the equation \eqref{eq:prop-cohomology} and the exact sequence \eqref{eq:relative-exact-sequence} that \(\omega\) is a section of \(j^{-1} \Omega_f^{n+p} \simeq \omega^{n+p}_{X \setminus \shC_f}\), where \(j : X \setminus \shC_f \longrightarrow X \) is the open embedding. If \(x \in X\) is a critical point of \(f\), we have that \(\omega\) restricted to \(X \setminus \shC_f\) is the image via \(\bar{c}_{f}\) of a unique section of \(j_* j^{-1} \Omega_f^{n+p}\). From the commutativity of the relative version of the diagram \eqref{eq:diagram-dualizing}, it follows that \(\omega\) is a section of \(\omega^{n + p}_f\). The first equation in \eqref{eq:prop-cohomology} implies that \(\omega\) is a section of \( \ker(\dd : \omega^{n + p-1}_f \rightarrow \omega_f^{n + p})\).

\jump

For the coboundaries, let \(\omega\) be a section of \(\shK^p\) mapping via the differential to a section of degree zero. Assume \(\omega\) has degree \(k > 0\) and let \(\omega_{\alpha}\) be a piece of \(\omega\) of degree exactly \(k = |\alpha| = |(\alpha_1, \dots, \alpha_{r+1})|\); namely, \(\omega_\alpha\) is a section of \(F_k \shK^p \setminus F_{k-1} \shK^p\) such that \(\overline{\omega} = \overline{\omega}_\alpha\) in \(\Gr^{F}_k \shK^p\). Around a regular point of \(f\), 
\[
    \underline{\dd}(\omega_\alpha) = \underline{\dd}\left(\left[\frac{\eta_\alpha}{x_1 \cdots x_r}\right] \otimes \partial_x^\alpha \right) = \left[\frac{\dd \eta_\alpha}{x_1 \cdots x_r}\right] \otimes \partial_{x}^\alpha + \sum_{i = 1}^{r+1} \left[ \frac{\dd x_i \wedge \eta_\alpha}{x_1 \cdots x_r} \right] \otimes \partial_{x_i} \partial_x^{\alpha}.
\]
Let us assume for now that \(\omega_\alpha\) is the only piece of the degree \(k\) of the coboundary \(\omega\). Hence, we must have 
\begin{equation} \label{eq:123}
    \left[\frac{\dd x_i \wedge \eta_\alpha}{x_1 \cdots x_r}\right] \otimes \partial_{x_i} \partial_{x}^\alpha = 0, \quad \textnormal{for} \quad i = 1, \dots, r+1.
\end{equation}
These imply that the local sections \([\dd x_i \wedge \eta_\alpha / x_1 \cdots x_r]\), \(i = 1, \dots, r+1\), of \(\shExt^r_{\Omod_U}(\Omod_X, \Omega^{r+1+p}_U)\) must be zero. But this is only possible if there exists \(\chi_\alpha\) a local section of \(\Omega^{p}_U\) such that \(\eta_\alpha = \dd x_1 \wedge \cdots \wedge \dd x_{r+1} \wedge \chi_\alpha\). Now, since we were assuming \(k > 0\), there is at least one coordinate \(\alpha_j > 0\) in \(\alpha\), so consider 
\[
\eta_{k,j} = \left[(-1)^j \frac{\dd x_1 \wedge \cdots \wedge \widehat{\dd x_j} \wedge \cdots \wedge \dd x_{r+1} \wedge \chi_k}{x_1 \cdots x_r} \right],
\]
a local section of \(\shExt^r_{\Omod_U}(\Omod_X, \Omega_U^{r+1+p})\). Subtracting \(\underline{\dd} (\eta_{\alpha, j} \otimes (\partial_x^\alpha)_j)\) from \(\omega\) we can assume that \(\omega\) has degree \(k - 1\). Here \((\partial_x^\alpha)_j\) is obtained from \(\partial_x^\alpha\) by replacing \(\alpha_j > 0\) with \(\alpha_j - 1\). This local construction produces a section \(\eta_{\alpha, j}\) of \(\shExt^r_{\Omod_U}(\Omod_X, \Omega_U^{r + 1 + p})\) on \(X \setminus \shC_f = X \setminus \shS_X\) that satisfies 
\[\dd F_j \wedge \eta_{\alpha, j} = \omega_\alpha, \qquad \textnormal{and} \qquad \dd F_\ell \wedge \eta_{\alpha, j} = 0, \quad \textnormal{for} \quad \ell \neq j.\]
It follows from \cref{lemma:extension} that \(\eta_{\alpha,j}\) extends to a global section on \(X\), still satisfying the same relations. 

\jump

If \(\omega_{\alpha}\) is not the only piece of maximal degree, we can proceed by lexicographical induction on the set of exponents \(\alpha\) such that \(|\alpha| = k\) as follows. Let \(\alpha \in \ZZ_{\geq 0}^{r+1}\) be the biggest exponent of degree \(k\) with respect to the lexicographical order such that \(\omega_\alpha \neq 0, |\alpha| = k\). If \(\alpha_i = 0\), for \(i = 1, \dots, j\), and \(\alpha_{j+1} > 0\), then \eqref{eq:123} holds for \(i = 1, \dots, j + 1\). Therefore, by the same construction above, we can find an equivalent section with \(\omega_\alpha = 0\) and with all degree \(k\) exponents lexicographically smaller than \(\alpha\).

\jump

Repeating the same argument a finite number of times, we can assume that \(\omega\) has degree zero. In this case the coboundary condition \(\underline{\dd} \omega = \omega_0\) translates, around the regular points of \(f\), into the conditions
\[
    \left[\frac{\dd \eta}{x_1 \cdots x_r}\right] = \left[\frac{\eta_0}{x_1 \cdots x_r}\right], \qquad \left[\frac{\dd x_i \wedge \eta}{x_1 \cdots x_r}\right] = 0, \quad \textnormal{for} \quad i = 1, \dots, r+1,
\]
which imply that \(\omega_0\) must be everywhere a section of \(\dd \omega_f^{n + p}\) by the same argument as before.

\jump

Let us now compute \(F_k \shH^p(\shK^\bullet)\) for \(k > 0\). Take \(\omega\) a section of \(\shK^\bullet\) such that \(\underline{\dd}(\omega) = 0\). We can always write \(\omega\) as
\[
\omega = \omega_0 + \omega_1 \otimes \partial_t + \cdots + \omega_k \otimes \partial_t^k,
\]
where \(\omega_i, i = 0, \dots, k\), are sections of \(\shH_X^r(\Omega^{d + p}_U)\). From \eqref{eq:diff-complex-K}, the differential \(\underline{\dd}\) on \(\shK^\bullet\) can be expressed in terms of the differential \(\dd\) of the complex \(\shH_X^r(\Omega^{d + \bullet}_U)\) as
\[
\underline{\dd}(\omega_i \otimes \partial_t^i) = {\dd}(\omega_i) \otimes \partial_t - (\dd F \wedge \omega_i) \otimes \partial_t^{i+1}.
\]
The cocycle condition on \(\omega\) implies that \(\dd \omega_0 = 0\). Hence, since Poincar\'e lemma holds, we have \(\omega_0 = \dd \eta_0\), for some \(\eta_0\) a section of \(\shH^r_{X}(\Omega^{d + p - 1}_U)\). Subtracting \(\underline{\dd} \eta_0\) from \(\omega\), we can assume \(\omega_0 = 0\). By induction, we can assume \(\omega_0, \dots, \omega_{k-1} = 0\). Using the same argument as with the coboundaries above, we can lower the degree of \(\omega_k\) without changing the class of \(\omega\) until \(\omega_k\) is a section of \(\ker(\omega_f^{n+p-1} \rightarrow \omega_f^{n+p})\).

\jump

Finally, assume
\[
\omega_k \otimes \partial_t^k = \underline{\dd} (\alpha_0 + \alpha_1 \otimes \partial_t + \cdots + \alpha_l \otimes \partial_t^l).
\]
Arguing as with the coboundaries in degree zero, we can assume \(\alpha_{k+1} = \cdots = \alpha_l = 0\). Using the Poincaré lemma as above, we can assume \(\alpha_0 = \cdots = \alpha_{k-1} = 0\), and repeating the coboundaries argument one last time, we can assume \(\alpha_k\) has degree zero. Consequently, \(\omega_k \otimes \partial_t^k = \underline{\dd} (\alpha_k \otimes \partial_t^k)\), and \(\omega_k \) is a section of \(\dd \omega_f^{n+p}\).

\jump

The last part of the proposition follows from the discussion in \cref{sec:regular-icis} and \cref{prop:cohomology-relative}.

\end{proof}

In contrast with the hypersurface case, \(F_0 \shH^0(\shK^\bullet) \simeq \shH^n(\omega_f^\bullet)\) is not stable under \(\partial_t^{-1}\). Conditions for the bijectivity of \(\partial_t^{-1}\) on \(\shH^n(\omega_f^\bullet)\) will be investigated in \cref{sec:rational}.


\begin{corollary} \label{cor:cohomology-derham}
   With the same assumptions above, 
   \[
        \shH^{-n}(\shK^\bullet) \simeq f^{-1} \Omod_S.
   \] 
\end{corollary}
\begin{proof}
    Let \(\omega\) be a section of \(\shH^r_X(\Omega^{r+1}_U)\). Then \(\underline{\dd} (\omega) = \dd \omega - (\dd F \wedge \omega) \partial_t \), implies that \((\dd F \wedge \omega) \partial_t\) equals \(\dd \omega\) on \(\shH^{-n}(\shK^\bullet)\). Hence, the class of \(\dd \omega\) must lie on \(\shH^0(\omega^\bullet_f)\) and consequently must be zero. By \eqref{eq:short-exact-seq}, we can identify \(\omega^0_f \simeq \Omega^0_f\) and \cite[8.16]{Loo84} implies
    \[\Omega^0_f \simeq \dd f \wedge \Omega_X^0.\]
    Therefore, the action of \(\partial_t\) on \(\shH^0(\omega^\bullet_f)\) is zero and \(\shH^{-n}(\shK^\bullet) \simeq \shH^0 (\omega^\bullet_f)\). The result then follows from the fact that \(\shH^0(\Omega^\bullet_f) \simeq f^{-1} \Omod_S\), see \cite[8.20]{Loo84}.
\end{proof}

\subsection{Coherence and regularity} \label{sec:coh-reg}

In this section, we will investigate the coherence and regularity of the Gauss-Manin system \(\shH^0_f\). Through this section, let \(f : X \longrightarrow S\) be a good representative of a smoothing of a positive-dimensional ICIS \((Z, x)\). We have the following lemma,

\begin{lemma} \label{lemma:stalks}
With these assumptions, we have the isomorphism
\[
\shH^0_f \simeq \shH^0(f_* \shK^\bullet).
\]
\end{lemma}
\begin{proof}
This result follows from combining the first spectral sequence of hypercohomology
\[
    'E_1^{p,q} = R^q f_* \shK^p \Longrightarrow \derR^{p+q} f_* \shK^\bullet,
\]
with \cref{prop:cohomology-derham} and \cref{cor:cohomology-derham} using the same argument as in \cite[3.3]{SS85}, which we will not repeat here.
\end{proof}

\begin{proof}[Proof of \cref{thm:main}]
As a consequence of \cref{lemma:stalks}, and the computations already seen in \cref{sec:cohomology}, we have the isomorphism of \(\Dmod_S\)-modules
\[
\shH_f^0 \simeq \shH^n(f_* \omega^\bullet_f) \otimes_{\CC} \CC[\partial_t],
\]
and the order-Ext filtration \(F_\bullet\) on \(\shH_f^0\), see \cref{sec:filtrations}, coincides with the filtration by the order of \(\partial_t\). Moreover, this filtration is generated at level zero. Notice that, from the inclusion \eqref{eq:inclusion}, the same is true for the Hodge filtration \(F^H_\bullet\).
\end{proof}

The coherence of \(F_k \shH^0_f\) as a \(\Omod_S\)-module would follow from the coherence of \(\shH^n(f_*\omega_f^\bullet)\). If \(f : X \longrightarrow S\) were proper, this would follow directly from Grauert's coherence theorem. This result could be exploited as in \cite[2.1]{Bri70} by using a globalization theorem for ICIS due to Looijenga \cite{Loo86}. However, it is more straightforward to use the coherence result from \cite[2.1.1]{BG80}. This result applies to the direct image by a deformation of an isolated singularity of any coherent complex of sheaves behaving like \(\Omega^\bullet_f\). Since the hypothesis \((P_1)-(P_4)\) in \textit{loc.~cit.} are certainly satisfied by \(\omega^\bullet_f\), one has that

\begin{lemma}
 The sheaves \(\shH^p(f_* \omega^\bullet_f)\) are coherent \(\Omod_S\)-modules and \(\shH^p(f_* \omega^\bullet_f)_0 \simeq H^p(\omega_{f,x}^\bullet)\).
\end{lemma}

Thanks to this, we finally obtain that the filtration \(F_\bullet\) on the Gauss-Manin system \(\shH^0_f\) is a good filtration generated at level zero and hence, \(\shH^0_f\) is a coherent \(\Dmod_S\)-module. Consequently, since \(S\) is one dimensional, \(\shH^0_f\) is moreover holonomic, see \cite[5.1.20]{HTT08}.

\jump

Recall that a \(\Dmod_S\)-module \(\shM\) is regular if it is finite over \(\Dmod_S\) and the localization \(\shM[t^{-1}]\) by \(t\) is a regular meromorphic connection. The regularity of \(\shH^0_f\) follows directly from the regularity of the Gauss-Manin connection of an ICIS due to Greuel \cite[3.6]{Gre75}. Indeed, the coherent \(\Omod_S\)-module \(\shH^n (\omega_f^\bullet)\) is a lattice of the \(\Dmod_S\)-module \(\shH^0_f\). Thus, localizing at \(t = 0\) we have the isomorphism
\begin{equation} \label{eq:localization}
\shH^0_f \otimes_{\Omod_S} \Omod_S[t^{-1}] \simeq \shH^n(f_* \omega_f^\bullet) \otimes_{\Omod_S} \Omod_S[t^{-1}]
\end{equation}
as \(\Dmod_S\)-modules. Since \(\shH^0(f_* \omega_f^\bullet)\) coincides with the lattice \(\shH'''\) from \cite[\S 2.2]{Gre75} the regularity of \eqref{eq:localization} follows from \textit{loc.~cit.} In summary

\begin{lemma} \label{lemma:regular-holonomic}
\(\shH^0_f\) is a regular holonomic \(\Dmod_S\)-module.
\end{lemma}

\subsection{Period integrals and \texorpdfstring{\(V\)}{V}-filtration} \label{sec:vfilt}

In the same settings as in previous sections, we have and identification between the restriction of \(\shH^p(f_* \omega_f^\bullet)\) to \(S^*\) and the sections of the cohomological Milnor fibration \(H^p(X_t, \CC), t \in S^*\). This identification can be realized by the non-degenerate pairing
\begin{equation} \label{eq:pairing}
    \restr{\shH^p(f_* \omega^\bullet_{f})}{X_t} \times H_p(X_t, \CC) \longrightarrow \CC
\end{equation}
given by integrating any representative of a class in \(\shH^n(f_* \omega^\bullet_{f})\) along a cycle in \(H_p(X_t, \CC)\). Since this pairing is trivial for \(p < n\), we will focus only on the case \(p = n\). For \(\gamma(t) \in H_n(X_t, \CC)\) a vanishing cycle, i.e. \(\gamma(t) \rightarrow 0\) as \(t \rightarrow 0\), and \(\omega \) any section of the dualizing sheaf \(\omega_{X}^{n+1}\), define the period integrals on the Milnor fiber by
\begin{equation} \label{eq:integrals}
I(t) = \int_{\gamma(t)} \frac{\omega}{\dd f},
\end{equation}
which are well-defined holomorphic functions on the universal cover of \(S^*\). Fixing an embedding of \(X\) in a germ of a manifold \(U\), and letting \(F_1, \dots, F_r \) be generators for \(X\), and \(\restr{F}{X} = f\), the form \(\omega/\dd f\) is nothing but the Gelfand-Leray form \(\restr{\eta / \dd F_1 \wedge \cdots \wedge \dd F_r \wedge \dd F}{X_t}\), for some section \(\eta \) of \(\Omega^{n+r+1}_{U}\). 

\jump

If we assume that \(\gamma(t)\) is a locally constant section of the fibration \(H_p(X_t, \CC), t \in S^*,\) for the (homological) Gauss-Manin connection, Leray's residue theorem \cite[\S III]{Pham11} imply
\[
    \frac{\dd}{\dd t} \int_{\gamma(t)} \frac{\omega}{\dd f} = \frac{\dd}{\dd t} \left(\frac{1}{2 \pi \imath} \int_{\delta \gamma(t)} \frac{\omega}{f - t} \right) = \frac{1}{2\pi \imath} \int_{\delta \gamma(t)} \frac{\omega}{(f-t)^2} = \int_{\gamma(t)} \partial_t\Big(\frac{\omega}{\dd f}\Big),
\]
where \(\delta \gamma(t)\) denotes the image of \(\gamma(t)\) under Leray's coboundary operator. The period integrals are therefore solutions to the (homological) Gauss-Manin connection. By the duality given by the pairing \eqref{eq:pairing}, the growth of solutions in the Brieskorn lattice \(\shH^n(f_*\omega^\bullet_f)\) can be controlled by the rate of growth of \(I(t)\) near \(t = 0\). 

\jump

By \cref{lemma:regular-holonomic}, \(\shH^0_f\) has regular singularities and the functions \(I(t)\) must be in the Nilsson class. Together with the Monodromy theorem, see for instance \cite[5.14]{Loo84}, this yields the asymptotic expansion
\begin{equation} \label{eq:expansion}
    I(t) = \int_{\gamma(t)} \frac{\omega}{\dd f} = \sum_{\alpha \in \Lambda} \sum_{0 \leq k \leq n} a_{\alpha, k}(\omega) t^\alpha (\log t)^k,
\end{equation}
where \(\Lambda = \{ \alpha \in \QQ\ |\ e^{-2 \pi i \alpha}\ \textnormal{eigenvalue monodromy}\} \) and \(a_{\alpha,k}(\omega) = 0\) if \(\alpha \ll 0\). If \(X\) is smooth, Malgrange \cite[4.5]{Mal74a} showed that \(a_{\alpha, k}(\omega) = 0 \) for \(\alpha \leq -1\).

\jump

The above motivates the introduction of the \(V\)-filtration on the \(\Dmod_{S}\)-module \(\shH^0_{f}\) along \( t = 0\), or more generally, on any regular holonomic \(\Dmod\)-module \(\shM\) on a complex manifold \(X\) along a submanifold \(Y\) of codimension 1 given locally by \(t = 0\). Following \cite[\S 3.1]{Sai88}, let
\[
V^m \Dmod_X = \{ P \in \Dmod_{X}\ |\ P \cdot (t)^j \subseteq (t)^{j + m},\ \textnormal{for all}\ j \in \ZZ \},
\]
with the convention that \((t)^j = \Omod_X\), if \(j \leq 0\). Then the \(V\)-filtration of Malgrange-Kashiwara on \(\shM\) along \(Y\) is a rational, exhaustive, decreasing, discrete, and left continuous filtration by coherent \(V^0 \Dmod_X\)-modules \(V^\alpha \shM\) such that for every \(\alpha \in \QQ\),
\begin{enumerate}
    \item We have the inclusion \(t \cdot V^\alpha \shM \subseteq V^{\alpha + 1} \shM\).
    \item We have the inclusion \(\partial_t \cdot V^\alpha \shM \subseteq V^{\alpha-1} \shM \).
    \item If \(V^{>\alpha} \shM = \bigcup_{\alpha' > \alpha} V^{\alpha'} \shM\), then \(t \partial_t - \alpha\) acts nilpotently on
    \[
        \Gr_V^\alpha \shM = V^\alpha \shM / V^{>\alpha} \shM.
    \]
\end{enumerate}
From these properties, it can be easily deduced that the maps
\[
 \Gr_V^\alpha \shM \xrightarrow{\,\partial_t\cdot\,} \Gr_V^{\alpha - 1} \shM \xrightarrow{\,t\cdot\,} \Gr_V^\alpha \shM
\]
are isomorphism of \(\Dmod_Y\)-modules for \(\alpha \neq 0\). This implies that the morphism \(t : V^\alpha \shM \longrightarrow V^{\alpha + 1} \shM \) is bijective for \(\alpha > -1\). If moreover one has that \(t\) is invertible on \(\shM\), then \(t \cdot V^\alpha \shM = V^{\alpha + 1} \shM\), for all \(\alpha \in \QQ\).

\jump

The following lemma will be useful later and follows directly from the classification of regular holonomic \(\Dmod\)-modules in one variable, see for instance \cite[19]{BM84} or \cite[\S 1]{Sai89}. See also \cite[7.1]{Hert01}.

\begin{lemma} \label{lemma:microlocal}
    If \(X\) is one dimensional, \(\partial_t : V^{\alpha} \shM \longrightarrow V^{\alpha -1 } \shM \) is bijective and \(V^{\alpha-1} \shM\) is a \(\CC\{\{\partial_t^{-1}\}\}\)-module for \(\alpha > 0\). In particular, the morphism \(\partial_t : V^{>0} \shM \longrightarrow V^{>-1} \shM\) is bijective and \(V^{>-1}\) is a \(\CC\{\{\partial_t^{-1}\}\}\)-module.
\end{lemma}

Here \(\CC\{\{\partial_t^{-1}\}\}\) denotes the ring of microlocal differential operators with constant coefficients, namely
\[
    \CC\{\{\partial_t^{-1}\}\} := \bigg\{ \sum_{k \geq 0} a_k \partial_t^{-1} \in \CC[[\partial^{-1}_t]]\, \bigg|\, \sum_{k\geq 0} \frac{a_k}{k!} t^k \in \CC\{t\} \bigg\}.
\]

In terms of the \(V\)-filtration, Malgrange's bound on the growth of \(I(t)\) when \(X\) is smooth can be expressed as
\begin{equation} \label{eq:growth}
    \shH^n(f_* \omega_{f}^\bullet) \subset V^{>-1} \shH^0_f.
\end{equation}

\begin{remark} \label{rmk:malgrange}
    The bound \eqref{eq:growth} on the growth rate of \(I(t)\) when \(X\) is smooth can be generalized when the singularities of \(Z\) are not isolated and the limiting cycle satisfy \(\gamma(0) \subset Z_{\textnormal{sing}}\). Using an integration by parts argument as in \cite[3.2]{Mal74b}, one can compare the leading exponents of the expansions \eqref{eq:expansion} with the roots of the Bernstein-Sato polynomial \(b_{f,x}(-s - 1)\). The roots of \(b_{f,x}(s)\) are always strictly negative rational numbers \cite[5.2]{Kas76}, so \eqref{eq:growth} holds in this more general setting.
\end{remark}

\subsection{Rational and Du Bois singularities} \label{sec:rational}

Let \(X\) be an irreducible analytic space of dimension \(n + 1\). A log resolution \(\pi : \widetilde{X} \longrightarrow X\) is a proper birational morphism, with \(\widetilde{X}\) smooth, and such that the exceptional locus \(E\) of \(\pi\) is a simple normal crossing divisor. An analytic space \(X\) has rational singularities if for some (hence any) log resolution, the canonical morphism
\[
\Omod_X \longrightarrow \derR \pi_* \Omod_{\widetilde{X}}
\]
is a quasi-isomorphism, that is, \(R^p \pi_* \Omod_{\widetilde{X}} = 0\), for \(p > 0\). In particular, \(X\) is normal. By Grauert-Riemenschneider vanishing \cite{GR70}, rational also implies Cohen-Macaulay. This gives the following equivalent characterization: \(X\) has rational singularities if the trace morphism
\begin{equation} \label{eq:trace}
    \pi_* \omega_{\widetilde{X}}^{n+1} \longrightarrow \omega^{n+1}_X
\end{equation}
is an isomorphism, see \cite[\S I.3]{Kem73}.

\jump

A smoothing \(f : (X, x) \longrightarrow (\CC, 0)\) of an ICIS will be called rational if \((X, x)\) has at most a rational singularity at \(x \in X\). It follows from \cite[2]{Elk78} that any smoothing of a rational singularity is a rational smoothing. The converse is not true; any hypersurface singularity has a rational smoothing \(f : (\CC^{n+1}, x) \longrightarrow (\CC, 0)\) given by its defining equation.

\begin{proposition}[\ref{prop:microlocal}] 
    If \(f : (X, x) \longrightarrow (\CC, 0)\) is a rational smoothing of an ICIS \((Z, x)\), then the Brieskorn lattice \(\shH^n(f_* \omega_f^\bullet)\) has structure of \(\CC\{\{\partial_t^{-1}\}\}\)-module.
\end{proposition}
\begin{proof}
Fix any log resolution \(\pi : (\widetilde{X}, E) \longrightarrow (X, x)\) of \((X, x)\), and let \(\tilde{f} = f \circ \pi\). Denote \(\widetilde{Z}\) the total transform of \(Z\), now with \(\widetilde{Z}_{\textnormal{sing}} = E\). Considering the period integrals \(I(t)\) from \eqref{eq:integrals}, one has
\[
    I(t) = \int_{\gamma(t)} \frac{\omega}{\dd f} = \int_{\pi_*^{-1}\gamma(t)} \frac{\pi^* \omega}{\dd \tilde{f}},
\]
since \(\pi\) being an isomorphism on \(X \setminus \{x\}\) implies \(\pi^* ({\omega / \dd f}{}) = {\pi^*\omega / \dd \tilde{f}}{}\) on the fibers \(X_t \simeq \widetilde{X}_t, t \in S^*\). If we denote \( \widetilde{\gamma}(t) = \pi_*^{-1} \gamma(t)\), the limiting cycle \(\widetilde{\gamma}(0)\) lives in \(\widetilde{Z}_{\textnormal{sing}}\). Using that \((X, x)\) has at most rational singularities, \eqref{eq:trace} gives that \(\pi^* \omega\) is always a section of \(\omega_{\widetilde{X}}^{n+1}\), and because \(\widetilde{X}\) is smooth, \(\omega_{\widetilde{X}}^{n+1} \simeq \Omega^{n+1}_{\widetilde{X}}\). Therefore, we can deduce that 
\[
    \shH^n(f_* \omega^\bullet_f) \subset V^{>-1} \shH^0_f,
\]
from Malgrange's bound \eqref{eq:growth} in the smooth case, see \cref{rmk:malgrange}. From \cref{lemma:microlocal} we have the morphism \(\partial_t^{-1} : V^{>-1} \shH^0_f \longrightarrow V^{>0} \shH^0_f\), which induces the action of \(\partial_t^{-1}\) on the Brieskorn lattice \(\shH^n(f_* \omega_f^\bullet)\).
\end{proof}

Du Bois singularities can be considered as a generalization of rational singularities. The general definition of Du Bois singularities due to Steenbrink \cite{Steen80} involves the Du Bois complex \(\underline{\Omega}_X^\bullet\) of \(X\), see \cite{DuBo81}. Alternatively, if \(X\) is normal and Cohen-Macaulay one has the following characterization from \cite[1.1]{KSS10}: \(X\) has Du Bois singularities if the morphism
\[
    \pi_* \omega^{n+1}_{\widetilde{X}}(E) \longrightarrow \omega^{n+1}_{X}
\]
is an isomorphism, \emph{cf.} \eqref{eq:trace}. Under these assumptions, it is straightforward that rational singularities are Du Bois. This implication holds in full generality, see \cite{Kov99}.

\jump

Recall the inclusion between the Hodge and Ext filtration on local cohomology from \eqref{eq:inclusion}. It was shown in \cite[C]{MP22} that for a pure dimensional Cohen-Macaulay singularity, 
\[F_0^H \shH^r_X(\Omod_U) = E_0 \shH^r_X (\Omod_U),\]
if and only if \(X\) is Du Bois. As before, we can call a smoothing \(f : (X,x) \longrightarrow (\CC,0)\) Du Bois if \((X, x)\) has at most a Du Bois singularity at \(x \in X\). As with the case of rational singularities, Du Bois singularities also deform \cite{KS16}. Hence,

\begin{proposition}[\ref{prop:hodge}]
 If \(f : (X, x) \longrightarrow (\CC, 0)\) is a Du Bois smoothing of a positive-dimensional ICIS \((Z, x)\), then 
 \[
     F^H_0 \shH^0_f  \simeq \shH^n(f_* \omega^\bullet_f ).
 \]
\end{proposition}

When \((Z, x)\) is an ICIS with Du Bois singularities, one has that any smoothing \(f : (X, x) \longrightarrow (\CC, 0)\) is not only Du Bois but also rational. This follows from the version of inversion of adjunction in \cite[5.1]{Sch07}. Therefore, assuming that \((Z, x)\) has Du Bois singularities is enough for both \cref{prop:microlocal} and \cref{prop:hodge} to hold.

\subsection{\texorpdfstring{\(b\)}{b}-functions} \label{sec:bfunc}

Using the notations from \cref{sec:GM}, let \( f : X \longrightarrow S \) be a good representative of a smoothing of an ICIS \((Z, x)\), where \(X\) is an ICIS embedded in a manifold \(U\) of dimension \(d = n + r + 1\). As usual, let \(F_1, \dots, F_r\) be a regular sequence defining \(X\) on \(U\) and choose \(F \) such that \(\restr{F}{X} = f\). Since \(f\) defines and smoothing, the local cohomology module 
\[
\iota_{+} \shH_X^r (\Omod_U) \simeq \shH^{r+1}_{\Gamma_f}(\Omod_{U\times S})
\]
from \cref{lemma:graph-embedding} is a cyclic \(\Dmod_{U}\langle t, \partial_t \rangle\)-module generated by the class \({\delta}_t = [1/F_1 \cdots F_r (F - t)]\). Since \(X\) is a complete intersection and \(\Gamma_t = \{ F - t = 0 \} \cap X\), one has the isomorphism
\[
\shH^{r+1}_{\Gamma_f}(\Omod_{U \times S}) \simeq \shH^1_{F-t}(\shH^r_{X}(\Omod_{U\times S}))
\]
so this module is isomorphic to \(\Dmod_{U}\langle t, \partial_t\rangle \delta \delta_t\), where \(\delta = [1/F_1 \cdots F_r]\) and \(\delta_t = [1/(F-t)]\). It is well-known that \(\delta_t\) and \(F^s\) satisfy the same relations, so \(\iota_{+} \shH^r_X(\Omod_U)\) is in turn isomorphic as \(\Dmod_{U}\langle t, \partial_t \rangle\)-modules to 
\[ \shN_f = \Dmod_U \delta F^s \subseteq \Omod_U[s, F^{-1}] F^s,\] 
where the action of \(t\) is given by multiplication by \(F\) and \(s \mapsto -\partial_t t\). The notation for \(\shN_f\) is justified by the fact that choosing different extensions \(F\) of \(f\) to \(U\) gives rise to isomorphic \(\Dmod_{U}\langle t, \partial_t \rangle\)-modules.

\jump

Considering the submodule \(\shM_f = \Dmod_U[s] \delta F^s \subseteq \shN_f\) leads to the (local) \(b\)-function \(b_{\delta, f}(s) \in \CC[s] \) defined by the functional equation
\[
P(s) \delta F^{s+1} = b_{\delta, f}(s) \delta F^s,
\]
where \(P(s) \in \Dmod_{U, x}[s]\), see for instance \cite[2.7]{Kas78}. One can directly check that \(b_{\delta, f}(s)\) only depends on \(f\). Since \(F\) is not a unit in \(\Omod_{U,x}\) and \( F \not\in (F_1, \dots, F_r)\Omod_{U,x}\), \(-1\) is always a root of \(b_{\delta, f}(s)\). Denote \(\tilde{b}_{\delta, f}(s) = b_{\delta, f}(s) / (s + 1)\). This \(b\)-function was studied by Torrelli in \cite{Tor02}. 

\jump

In \cite[1.8]{Tor09} a third \(b\)-function sitting between \(b_{\delta, f}(s)\) and \(\tilde{b}_{\delta, f}(s)\) is introduced. Let \(\widehat{b}_{\delta, f}(s) \in \CC[s] \) be the minimal polynomial satisfying 
\[
\widehat{b}_{\delta, f}(s) \delta F^s \in \Dmod_{U,x} (J_{\underline{F}, F}, F) \delta F^s,
\]
where \(J_{\underline{F}, F}\) is the ideal generated by the \((r+1)\)-minors of the Jacobian matrix of \(\underline{F} = (F_1, \cdots, F_{r+1})\) and \(F\). As shown in \cite[1.9]{Tor09}, the polynomial \(\widehat{b}_{\delta, f}(s)\) is either equal to \(b_{\delta, f}(s)\) or \(\tilde{b}_{\delta, f}(s)\). However, it is not clear when the equality \(\widehat{b}_{\delta, f}(s) = \tilde{b}_{\delta, f}(s)\) must hold, see the discussion in \cite[\S 3]{Tor09}. 

\jump

We will next relate the polynomial \(\widehat{b}_{\delta, f}(s)\) with the Brieskorn lattice \(H''_{f,x}\) following \cite{Mal75}.

\begin{lemma} \label{lemma:origin}
The \(\Dmod_{U,x}\)-module \(\shN_{f,x} / \shM_{f,x}\) is supported at the origin.
\end{lemma}
\begin{proof}
    Since the module \(\shN_{f,x}\) is generated by \( \partial_t^k \delta F^s\) over \(\Dmod_{U,x}\), it is enough to show that there exists \(\ell \geq 0\) such that \(x_i^{\ell} \partial_t^k \delta F^s\) is in \(\shM_{f,x}\), for \(i = 1, \dots, d\). The result is clear for \(k = 0\), so we will argue by induction on \(k \in \mathbb{Z}_{\geq 0}\).

    \jump

    By hypothesis, the critical locus \(\shC_{f,x}\) of \(f\) is isolated, hence \(x_i^n \in (J_{\underline{F}, F}, F_1, \dots, F_r) \Omod_{U, x}\), for some \(n \in \mathbb{Z}_{\geq 0}\). Let us introduce some notation. By \(m_{k_0, \dots, k_r}(\underline{H})\) we will denote the minor of the Jacobian matrix \(J_{\underline{H}}\) with respect to the columns \( k_0 < k_1 < \cdots < k_{r} \). Then, we have the identity
    \begin{equation} \label{eq:diff-op}
        \left[\sum_{i = 0}^r (-1)^{r+i} m_{k_0, \dots, \hat{k}_i, \dots, k_r}(\underline{F}) \frac{\partial}{\partial x_i}\right] \cdot \delta F^s = - m_{k_0, \dots, k_r}(\underline{F}, F) \partial_t \delta F^s,
    \end{equation}
    see \textit{loc.~cit.} Denote by \(\Delta_{k_0, \dots, k_{r}}(\underline{F})\) the differential operator on the left-hand side of \eqref{eq:diff-op}.
    
    \jump

    By induction hypothesis there exists \(\ell \geq 0\) such that \(x_i^{\ell} \partial_t^k \delta F^s \in \shM_{f,x}\), for \(i = 1, \dots, d\). Using \eqref{eq:diff-op}, one has
    \begin{equation*}
    \begin{split}
        x_i^{\ell+1} m_{k_0, \dots, k_r}(F, \underline{F}) \partial_t^{k+1} \delta F^s & = - x_i^{\ell+1} \Delta_{k_0, \dots, k_r}(\underline{F}) \partial_t^k \delta F^s \\ & = -\Delta_{k_0, \dots, k_r}(\underline{F}) x_i^{\ell+1} \partial_t^k \delta F^s - [\Delta_{k_0, \dots, k_r}(\underline{F}), x_i^{\ell+1}] \partial_t^k \delta F^s.
    \end{split} 
    \end{equation*}
    Since \(\Delta_{k_0, \dots, k_r}(\underline{F})\) is a derivation, the above expression belongs to \(\shM_{f,x}\) by induction hypothesis. Therefore, the induction step is proven: \(x_i^{\ell+n+1} \partial_t^{k+1} \delta F^s \in \shM_{f,x}\), for \(i = 1, \dots, r\), since \(F_j \delta F^s = 0 \), for \( j = 1, \dots r\).
\end{proof}

Consider now the short exact sequence of \(\Dmod_{U,x}\langle t, \partial_t \rangle\)-modules,
\[
0 \longrightarrow \shM_{f,x} \longrightarrow \shN_{f,x} \longrightarrow \shN_{f,x} / \shM_{f,x} \longrightarrow 0.
\]
Taking cohomology of the de Rham complexes over \(U\), leads to the long exact sequence of \(\CC\{t\}\langle \partial_t \rangle\)-modules
\[
\cdots \longrightarrow H^p(\DR_U \shM_{f,x}) \longrightarrow H^p(\DR_U \shN_{f,x}) \longrightarrow H^p(\DR_U (\shN_{f,x} / \shM_{f,x})) \longrightarrow \cdots.
\]
Identifying \(\DR_U \shN_{f,x}\) with the stalk of at \(x\) of the complex \(\shK^\bullet\) from \eqref{eq:complex-K}, we have that \(H^p(\shN_{f,x}) = 0\), for \(p \neq -n, 0\), by \cref{prop:cohomology-derham}. \cref{lemma:origin} implies \(H^p(\shN_{f,x} / \shM_{f,x}) = 0 \), for \(p \neq 0\), see \cite[2.10]{Mal75}. All this yields the injection
\[
H^0(\DR_U \shM_{f,x}) \xhookrightarrow{\ \quad\ } H^0(\DR_U \shN_{f,x}) \simeq H''_{f,x}[\partial_t],
\]
where \(H''_{f,x}\) is the Brieskorn lattice from \cref{prop:cohomology-relative}. By regularity of the Gauss-Manin connection, the image of \(H^0(\DR_U \shM_{f,x})\) coincides with the saturation of \(H_{f,x}''\),
\begin{equation} \label{eq:saturation}
    \widetilde{H}''_{f,x} = \sum_{k \geq 0} (t \partial_t)^k H_{f,x}'' = \sum_{k \geq 0} (\partial_t t)^k H_{f,x}''.
\end{equation}

\begin{proof}[Proof of \cref{thm:malgrange}]
    The polynomial \(\widehat{b}_{\delta, f}(s)\) is the minimal polynomial of the action of \(s\) on the module
    \[
        L_{f,x} = \frac{\Dmod_{U,x}[s] \delta F^s}{\Dmod_{U,x}[s](J_{\underline{F}, F}, F) \delta F^s}.
    \]
    Because \(\shM_{f,x}/t \shM_{f,x}\) is finitely generated over \(\Dmod_{U,x}\) so is \(L_{f,x}\). Moreover, \(L_{f,x}\) is supported at the origin. Therefore, \(\textnormal{End}_{\Dmod_{U,x}}(L_{f,x}) \simeq \textnormal{End}_{\CC}(H^0(\DR_U L_{f,x}))\) and \(\widehat{b}_{\delta,f}(s)\) is the minimal polynomial of \(s\) in \(H^0(\DR_U L_{f,x})\).
    
    \jump

    To finish the proof, it is enough to show that the morphism
    \begin{equation} \label{eq:last-morphism}
        H^0(\DR_U(\shM_{f,x} / t \shM_{f,x})) \simeq {\widetilde{H}''_{f,x}}/{t \widetilde{H}''_{f,x}} \longrightarrow  H^0(\DR_U L_{f,x})
    \end{equation}
    is an isomorphism. Since \(H^0(\DR_U(\shF_x)) = \shF_x \otimes_{\Dmod_{U,x}} \Omega^d_{U,x}\), for any \(\Dmod_{U,x}\)-module \(\shF_x\), see \cite[2.1]{Mal75}, the morphism \eqref{eq:last-morphism} is surjective, and it is enough to show that \( J_{\underline{F}, F}\Omega^d_{U,x}[s] \delta F^s \) is zero in \( \widetilde{H}_{f,x}'' / t \widetilde{H}''_{f,x} \).
    
    \jump

    First, notice that if we denote \(\widetilde{F}_x\) the saturation of \(\partial^{-1}_t H_{f,x}''\), then \(t \widetilde{H}_{f,x}'' \simeq \widetilde{F}_x\). Indeed, by definition of saturation \((t\partial_t) \widetilde{F}_x \subseteq \widetilde{F}_x\). Since \(\partial_t : \partial_t^{-1} H_{f,x}'' \longrightarrow H''_{f,x}\) is a bijection, \(\partial_t \widetilde{F}_x \simeq \widetilde{H}_{f,x}''\). Moreover, by \cref{prop:microlocal} and the properties of the \(V\)-filtration, the action of \(t\) in \(H''_{f,x}\) is bijective. Therefore, \(t \widetilde{H}_{f,x}'' \simeq \widetilde{F}_x\).
    
    \jump

    Let us take an element \(\omega \in J_{\underline{F}, F} \Omega^d_{U,x} \delta\), which can be identified with an element of \(\textnormal{Ext}_{\Omod_{U,x}}^r(\Omod_{X,x}, \dd F \wedge \dd F_1 \wedge \cdots \wedge \dd F_r \wedge \Omega^n_{U,x})\). Writing \(\omega = \dd F \wedge \dd F_1 \wedge \cdots \wedge \dd F_r \wedge \eta\), with \(\eta \in \textnormal{Ext}^r_{\Omod_{U,x}}(\Omod_{X,x}, \Omega^{n}_{U,x})\), the induced class on the complex \(\shK^\bullet\) from \eqref{eq:complex-K} satisfies,
    \[
        \partial_t([\dd F \wedge \dd F_1 \wedge \cdots \wedge \dd F_r \wedge \eta]) = [\dd F_1 \wedge \cdots \wedge \dd F_r \wedge \dd \eta],
    \]
    with the right-hand side belonging to \(\omega_{f,x}^n/ \dd \Omega^{n-1}_{f,x}\). Thus, since the smoothing is rational, \(\partial_t\) is invertible on \(H''_{f,x}\) by \cref{prop:microlocal} and \(\omega \in \partial_t^{-1} H_{f,x}''\), so \(\omega \in \widetilde{F}_x \) as required. In general, let \(\omega(s) \in J_{\underline{F}, F} \Omega^d_{U,x}[s] \delta F^s\). From \eqref{eq:diff-op}, 
    \[
        (s+1) m_{k_0, \dots, k_r}(\underline{F}, F) \delta - \Delta_{k_0, \dots, k_r}(\underline{F}) F \delta \in \textnormal{Ann}_{\Dmod_{U,x}[s] \delta} F^s.
    \]
    Hence, \(m_{k_0, \dots, k_r}(\underline{F},F) \omega(s) = m_{k_0, \dots, k_r}(\underline{F}, F)\omega(-1) \) modulo \( \textnormal{Ann}_{\Dmod_{U,x}[s]\delta} F^s \cdot \Omega^d_{U,x}[s]\delta + F \cdot \Omega^d_{U,x}[s]\delta\). Finally, from \(H^0(\DR_U(\shM_{f,x})) \simeq \widetilde{H}_{f,x}'' \), one deduces the isomorphism
    \[
    \widetilde{H}_{f,x}'' / t \widetilde{H}_{f,x}'' \simeq \Omega^d_{U,x}[s]\delta / (\textnormal{Ann}_{\Dmod_{U,x}[s]\delta} F^s \cdot \Omega^d_{U,x}[s] \delta + F \cdot \Omega^d_{U,x}[s]\delta ),
    \]
    and any element \(m_{k_0, \dots, k_r}(\underline{F}, F) \omega(s) \) is equivalent to an element lying on \(t \widetilde{H}_{f,x}'' \simeq \widetilde{F}_x\). This concludes the proof.
\end{proof}

As a consequence, we obtain the following bound for the degree of the \(b\)-function \(b_{\delta,f}(s)\).

\begin{corollary}
    The degree of \(b_{\delta, f}(s)\) is bounded by \(\mu(Z, x) + 1\).
\end{corollary}

\cref{thm:malgrange} allows an algebraic computation of the eigenvalues of the monodromy of any smoothing \(f: (X, x) \longrightarrow (\CC, 0)\) of an ICIS; in particular, for the generic monodromy of an ICIS discussed in \cref{sec:generic-monodromy}. We end this section with the following example.

\begin{example} \label{ex:2}
    Consider \(Z \subseteq \CC^3 \), the complete intersection defined by \(I = (xy + z^3, x^2 + yz)\), having only one isolated singular point at the origin of \(\CC^3\).
    
    \jump
    
    Let \(f : (X, 0) \longrightarrow (\CC^3, 0) \) be a \(\mu\)-minimal smoothing of the ICIS \((Z, 0)\). The eigenvalues of the generic monodromy of \((Z, 0)\) can be calculated directly using the Gauss-Manin connection acting on the Brieskorn lattice \(H''_{f,0}\) by adapting Brieskorn's original algorithm \cite[\S 4.5]{Bri70}, see \cite{Tor24}. Alternatively, by \cref{thm:malgrange}, one can compute the \(b\)-function \(b_{\delta, f}(s)\) using \textsc{Macaulay2} \cite{M2}. In any case, all the eigenvalues of the generic monodromy of \((Z, 0)\) have the form \(e^{2 \pi \imath k/ 9}\). 
    
    \jump

    On the other hand, the Verdier monodromy of the variety \(Z\) on the corresponding nearby sheaf can be calculated through the Berstein-Sato polynomial \(b_{Z}(s)\) of the variety \(Z\), see \cite[2.8]{BMS06}. Using \textsc{Singular} \cite{Sing} one can verify that \(b_{Z}(s)\) has \(-3/2\) as a root. Hence, the Verdier monodromy of \(Z\) has an extra eigenvalue equal to \(e^{\pi \imath}\). 

    \jump

    This difference can be explained by the fact that the Verdier monodromy also captures the Picard-Lefschetz monodromy of a general element of the ideal \(I\). Indeed, if \(g = \alpha_1 f_1 + \alpha_2 f_2 \in I\) is a general linear combination of the generators,
    \[
        b_{g}(s) = (s+1)\left(s+\frac{3}{2}\right).
    \]
    In general, this fact can be made precise using \cite{Mus22}.
\end{example}

\bibliography{/Users/guillem/Dropbox/Maths/Bib/serials_short,/Users/guillem/Dropbox/Maths/Bib/references}
\bibliographystyle{amsalpha}

\end{document}